\documentclass{amsart}
\usepackage{fullpage}
\usepackage{amsmath}
\usepackage{amssymb}
\usepackage{mathabx}
\usepackage{todonotes}
\usepackage{tikz-cd}
\usepackage{mathtools}

\begin{document}

\newcommand{\+}{\boxplus}
\newcommand{\x}{\cdot}
\newcommand{\del}{\setminus}
\newcommand{\ignore}[1]{}
\newcommand{\semi}[2]{#1_{#2}}
\newtheorem{example}{Example}
\newtheorem{theorem}{Theorem}
\newtheorem{lemma}[theorem]{Lemma}
\newtheorem{corollary}[theorem]{Corollary}
\newtheorem{conjecture}{Conjecture}
\newtheorem{question}[conjecture]{Question}
\newtheorem{claim}{Claim}[theorem]
\newcommand{\Z}{\mathbb{Z}}
\newcommand{\N}{\mathbb{N}}
\newcommand{\one}{{\bf 1}}
\newcommand{\R}{\mathbb{R}} 
\newcommand{\supp}{\mbox{Supp}}
\newcommand{\minsupp}{\mbox{Min}}
\newcommand{\tl}{\uparrow}
\newcommand{\tlinv}{\triangledown}

\title{Perfect matroids over hyperfields}
\author{Nathan Bowler and Rudi Pendavingh}
\begin{abstract}
We investigate valuated matroids with an additional algebraic structure on their residue matroids. We encode the structure in terms of representability over stringent hyperfields. 

A hyperfield $H$ is {\em stringent} if $a\+b$ is a singleton unless $a=-b$, for all $a,b\in H$. By a construction of Marc Krasner, each valued field gives rise to a stringent hyperfield.

We show that if $H$ is a stringent skew hyperfield, then the vectors of any weak matroid over $H$ are orthogonal to its covectors, and we deduce that weak matroids over $H$ are strong matroids over $H$. Also, we present vector axioms for matroids over stringent skew hyperfields which generalize the vector axioms for oriented matroids and valuated matroids. 
\end{abstract}
\maketitle

\section{Introduction} 

Valuated matroids as introduced by Dress and Wenzel in \cite{DressWenzel1992a} play a fundamental role in tropical geometry, since they correspond precisely to tropical linear spaces \cite{Speyer2008}. Tropical linear spaces can be obtained as tropicalisations of classical linear subspaces. They play a key role in the study of realisabillity and intersection theory \cite{FrancoisRau2013} \cite{Rincon2013} \cite{Shaw2013}.

Each valuated matroid gives rise to a collection of ordinary matroids, which Dress and Wenzel called residue matroids but which are also often referred to as initial matroids. Together these residue matroids form a structure called a {\em matroid flock}. Indeed, such matroid flocks are cryptomorphic with valuated matroids as shown in \cite{BDP2018}. In recent years, a number of independent cases have emerged where each of the residue matroids has some additional structure. For example in \cite{Brandt2019}, Brandt investigates which valuated matroids arise from linear spaces over a given valued field $K$, and shows that the residue matroids of such valuated matroids are always representable over the residue field of $K$. Similary, in \cite{BDP2018}, Bollen, Draisma, and Pendavingh show that the residue matroids corresponding to the Lindstr\"om valuation of an algebraic matroid over a field $K$ are necessarily linear over $K$, and use this to rule out that certain matroids are algebraic over $K$. On the other hand, in \cite{Juergens2018}, J\"urgens shows that the residue matroids of the valuated matroids corresponding to real tropical linear spaces are all orientable. Furthermore, in each case these structures satisfy compatibility conditions analogous to those connecting the residue matroids into a matroid flock.

In this paper, we will provide a broader context for these results by outlining a theory of representation of matroids over algebraic objects called stringent hyperfields such that in each case the whole structure (valuation on a matroid together with extra structure on the various residue matroids, all satisfying compatibility conditions) can be encoded in terms of a representation of the matroid over a single stringent hyperfield. Hyperfields are variants of fields defined by Marc Krasner in \cite{Krasner1957} in which adding two elements may yield several elements rather than just one. Krasner used this construct to define extensions of the residue field of a valued field. We say a hyperfield $H$ is {\em stringent} if the hypersum  $a\+ b$ is  a singleton unless $a=-b$, for all $a,b\in H$. In addition we shall show that representation of matroids over stringent hyperfields has a number of desirable properties which existing unifying approaches to the representation of matroids over weak algebraic objects lack. 

The history of attempts to develop a general theory of matroid representation generalising matroids, oriented matroids and and linear spaces goes back decades. In \cite{Dress1986}, Dress defined {\em matroids with coefficients in a fuzzy ring} as a first such common abstraction. In this general theory,  each of these classes arises as matroids over a particular fuzzy ring. Dress and Wenzel later also defined {\em valuated matroids} within this framework \cite{DressWenzel1992a}.

Matroids over hyperfields were introduced by Baker and Bowler in \cite{BakerBowler2017}. Baker and Bowler obtain matroids over hyperfields as a special case of their more general theory of matroids over {\em tracts}. In this theory, the hyperfields and tracts play a role which is very similar to that of the fuzzy fields of Dress. It was shown by Giansiracusa, Jun, and Lorscheid \cite{GJL2017} that  there are canonical functors between the class of fuzzy rings and the class of hyperfields.  Via these functors, matroids with coefficients in a fuzzy ring and (strong) matroids over hyperfields are essentially equivalent notions. 

It seems fair to say that the theory of matroids with coefficients constantly aspires towards the condition of oriented matroids, with its broad variety of different, yet equivalent, axiom systems.  Among the many axiom systems for oriented matroids offered in \cite[Ch.3]{OM}, one can distinguish at least the following three types:
\begin{enumerate}
\item Grassmann-Pl\"ucker relations for chirotopes; orthogonality of circuits and cocircuits; 
\item 3-term Grassmann-Pl\"ucker relations for  chirotopes; local orthogonality of circuits and cocircuits;  modular circuit elimination axioms;
\item vector axioms.
\end{enumerate}
The Grassmann-Pl\"ucker relations for chirotopes generalize the symmetric base exchange axiom for ordinary matroids, and both are combinatorial shadows of the Grassmann-Pl\"ucker relations among the Pl\"ucker coordinates of a linear subspace. The second type of axioms are weaker, `local' versions of the first type.
The vector axioms  of an oriented  matroid closely resemble the definition of a linear subspace as a set of vectors closed under addition and scalar multiplication, and they refine the axiom system for the flats of an ordinary matroid. So the equivalence of these different axiomatizations holds true for matroids, oriented matroids and linear spaces, but also for valuated matroids. The equivalence of type (1) and (2) axioms  was established for valuated matroids by Dress and Wenzel \cite{DressWenzel1992a}, and Murota and Tamura   showed that valuated matroids are characterized by type (3) vector axioms \cite{Murota2001}.

However, the equivalence of the type (1) and  (2) axioms does not extend to the more general notions of representation mentioned above. Accordingly Baker and Bowler distinguish {\em strong } and {\em weak} matroids over hyperfields, based on axioms of type (1) and (2) respectively. 

Moreover, there seems to be no natural generalization of the vector axioms (3), if only because these axioms refer to a composition operation for signs which has no counterpart in general fuzzy rings or hyperfields. Anderson develops vector axioms for matroids over tracts in \cite{Anderson2019}, but her axioms do not closely resemble the vector axioms for oriented matroids; they give an external characterisation of which sets should count as vectors involving a quantification over all bases, rather than an internal characterisation of how to build new vectors from old.

The collections of vectors of oriented matroids behave very well with respect to minors. For example, the vectors of $M.X$ are simply the restrictions of the vectors of $M$ to $X$. In \cite{Anderson2019}, Anderson gives examples showing that this does not hold in general for matroids over hyperfields. Examples of the same phenomenon for matroids with coefficients in a given fuzzy ring could be constructed along similar lines.

The main objective of the present paper is to show that stringent skew hyperfields form a natural class of hyperfields which is broad enough to cover the examples mentioned above, as well as classical, oriented and valuated matroids, but such that representation over these hyperfields retains all the good properties of the theory of oriented matroids. In order to include the examples arising from algebraic matroids in \cite{Pendavingh2018}, we extend these results to the setting of skew hyperfields, in which the multiplication is not necessarily commutative.

This paper has two main results, both stated in terms of vectors and covectors of matroids over stringent skew hyperfields. If $M$ is a left $H$-matroid on ground set $E$, then $V\in H^E$ is {\em vector} of $M$ if $V\perp X$ for each circuit $X$ of $M$, and $U\in H^E$ is a {\em covector} of $M$ if $Y\perp U$ for each cocircuit $Y$ of $M$. 
\begin{theorem} \label{thm:intro1}Let $H$ be stringent skew hyperfield, and let  $M$ be a left $H$-matroid on ground set $E$. If $V$ is a vector of $M$ and $U$ is a covector of $M$, then $V\perp U$.\end{theorem}

In \cite{DressWenzel1992b}, Dress and Wenzel explored the class of {\em perfect} fuzzy rings $R$, which they defined as those such that, for any strong matroid over $R$, all vectors are orthogonal to all covectors. They showed that these matroids have the property that type (1) axioms are equivalent to type (2). They showed perfection of a significant class of fuzzy rings, which includes the ones required for defining classical, oriented, and valuated matroids. Adapting these results, Baker and Bowler argue that  over {\em doubly distributive} hyperfields, weak matroids are equivalent to strong matroids. 

Since if $R$ is perfect then any weak matroid over $R$ is strong, it follows that if $R$ is perfect then even the weak matroids over $R$ will have the property that all vectors are orthogonal to all covectors. In the current paper, we take this stronger statement as our definition of perfection. This makes it a little easier for us to show that if $R$ is perfect then the type (1) and type (2) axioms are equivalent.

Theorem \ref{thm:intro1} extends the existing results in two ways: to stringent hyperfields, which properly include the doubly distributive hyperfields;  to stringent  {\em skew} hyperfields even, which also generalize skew fields.

\begin{theorem} \label{thm:intro2}Let $H$ be a stringent skew hyperfield. Let $E$ be a finite set, and let $\mathcal{V}\subseteq H^E$. There is a left $H$-matroid $M$ such that $\mathcal{V}=\mathcal{V}(M)$ if and only if
\begin{enumerate}
\item[(V0) ] $0\in \mathcal{V}$.
\item[(V1) ] if $a\in H$ and $V\in \mathcal{V}$, then  $aV\in \mathcal{V}$.
\item[(V2)$'$] if $V, W\in \mathcal{V}(M)$ and $\underline{V\circ W}=\underline{V}\cup \underline{W}$, then $V\circ W\in\mathcal{V}(M)$. 
\item[(V3) ] if $V, W\in \mathcal{V}, e\in E$ such that $V_e=-W_e\neq 0$, then there is a $Z\in \mathcal{V}$ such that $Z\in V\+ W$ and $Z_e=0$.
\end{enumerate}
\end{theorem}
This theorem features a composition $\circ: H\times H\rightarrow H$, which we will define for all stringent hyperfields. For the tropical hyperfield, this composition is $a\circ b=\max\{a,b\}$, so that this theorem specializes to a similar characterization of Murota and Tamura \cite{Murota2001}. If $H$ is the hyperfield of signs, then $a\circ b=a$ if $a\neq 0$ and  
 $a\circ b=b$ otherwise, and then the theorem gives the oriented matroid vector axioms.
 
Composition operators as in (V2)$'$, and elimination properties such as (V3) were also discussed by Anderson in \cite{Anderson2019}.  

As a tool in establishing these results, we also show that the collections of vectors of matroids over stringent skew hyperfields behave well with respect to taking minors, so that for example the vectors of $M.X$ are as expected (and as for oriented and valuated matroids) the restrictions of the vectors of $M$ to $X$.
 
A main ingredient of our analysis is a recent classification of stringent skew hyperfields due to Bowler and Su \cite{BowlerSu2019}. 
 If $H$ is a stringent skew hyperfield, then by their work there exists a linearly ordered group $(\Gamma,<)$ and a multiplicative group homomorphism $\psi: H^\star \rightarrow \Gamma$ such that:
\begin{enumerate}
\item $\psi(x)>\psi(y)\Rightarrow x\+y=\{x\}$ for all $x,y\in H^\star$; and 
\item the restriction of $H$ to $R:=\{0\}\cup \{x:\psi(x)=1\}$ is the Krasner hyperfield, the sign hyperfield, or a skew field.
\end{enumerate}
One can think of the function $\psi$ as a (non-Archimedean) valuation of $H$, and the sub-hyperfield $R$ as its residue hyperfield. 

We will show in this paper that if $M$ is a matroid over a stringent skew hyperfield $H$ with residue $R$, then there exists a matroid $M_0$ over $R$ whose bases are a subset of the bases of $M$, and whose coefficients are essentially an induced subset of the coefficients of $M$. This {\em residue matroid} $M_0$  generalizes the residue matroid of a valuated matroid as defined by Dress and Wenzel. We prove our main theorems for  a matroid $M$ over a stringent skew hyperfield by applying well-known facts about matroids, oriented matroids and matroids over skew fields to residue matroids which arise from $M$.

The three cases for the residue of $H$ need similar, yet subtly different argumentation. We chose to present the case that the residue is the Krasner hyperfield separately in Section 3, and the general case in Section 4. The construction of the residue matroid is more involved for skew hyperfields compared to commutative hyperfields. We settle these difficulties in Section 3, so that the reader who is only interested in the commutative case could skip this section. Apart from  the construction of the residue matroid, the proofs of the main theorems for the special case in Section 3 are easier than for the general case, and they may serve as a stepping stone for the general case in Section 4.

\subsection*{Acknowledgement} This work started when both authors attended the CIRM workshop on Oriented Matroids in September 2018. We thank the institute for their hospitality and the organizers for inviting us.
At the workshop, Laura Anderson presented her work on vector axioms for matroids over hyperfields. We thank Laura for her presentation and for stimulating conversations during the workshop.

\section{Matroids over  hyperfields}
\subsection{Hyperfields}
A {\em hyperoperation} on $G$ is a map $\boxplus:G\times G\rightarrow 2^G$. A hyperoperation induces a map $\overline{\boxplus}:2^G\times 2^G\rightarrow 2^G$ by setting 
$$X ~\overline{\boxplus} ~Y:=\bigcup\{x\boxplus y: x\in X, y\in Y\}.$$
We write $x\+ Y:=\{x\}~\overline{\boxplus}~Y$, $X\+ y:=X~\overline{\boxplus}~\{y\}$, and $X\+Y:=X~\overline{\boxplus}~Y$.
The hyperoperation $\boxplus$ then is {\em associative}  if $x\+(y\+ z)=(x\+ y)\+z$ for all $x,y,z\in G$.

A {\em hypergroup} is a triple $(G, \boxplus, 0)$ such that $0\in G$ and $\boxplus: G\times G\rightarrow 2^G\del\{\emptyset\}$ is an associative hyperoperation, and
\begin{itemize}
\item[(H0)] $x\boxplus 0=\{x\}$
\item[(H1)] for each $x\in G$ there is a unique $y\in G$ so that $0\in x\boxplus y$, denoted $-x:=y$
\item[(H2)] $x\in y\boxplus z$ if and only if $z\in (-y)\boxplus x$
\end{itemize}

A {\em hyperring} is a tuple $(R, \cdot, \boxplus, 1, 0)$ so that
\begin{itemize}
\item[(R0)] $(R, \boxplus, 0)$ is a commutative hypergroup
\item[(R1)]  $(R^\star, \cdot, 1)$ is monoid, where we denote $R^\star:=R\setminus\{0\}$
\item[(R2)]  $0\cdot x=x\cdot 0=0$ for all $x\in R$
\item[(R3)]  $\alpha(x\boxplus y)=\alpha x\boxplus \alpha y$ and $(x\boxplus y)\alpha=x \alpha \boxplus y\alpha $ for all $\alpha,x,y\in R$
\end{itemize}
A {\em skew hyperfield} is a hyperring such that $0\neq 1$, and each nonzero element has a multiplicative inverse. A {\em hyperfield} is then a skew hyperfield with commutative multiplication.  

The following skew hyperfields play a central role in this paper:
\begin{itemize}
\item The {\em Krasner hyperfield} $\mathbb{K}=(\{0,1\}, \cdot, \boxplus, 1, 0)$, with hyperaddition $1\boxplus 1=\{0,1\}$. 
\item The {\em sign hyperfield} $\mathbb{S}=(\{0,1,-1\}, \cdot, \boxplus, 1, 0)$, with 
$$1\boxplus 1=\{1\}, ~-1\boxplus -1=\{-1\}, ~1\boxplus -1=\{0,1,-1\}$$ and the usual multiplication. 
\item Skew fields $K$, which can be considered as skew hyperfields with hyperaddition $x\+y=\{x+y\}$. 

\end{itemize}

If $G, H$ are hypergroups, then a map $f:G\rightarrow H$ is a {\em hypergroup homomorphism} if $f(x\+y)\subseteq f(x)\+f(y)$ for all $x,y\in G$, and $f(0)=0$. If $R, S$ are hyperrings, then $f:R\rightarrow S$ is a {\em hyperring homomorphism} if $f$ is a hypergroup homomorphism, $f(1)=1$, and $f(x\x y)=f(x)\x f(y)$ for all $x, y\in R$. 
A (skew) hyperfield homomorphism is just a homomorphism of the underlying hyperrings. 

\subsection{Matroids over hyperfields}
Let $H$ be a skew hyperfield, and let $E$ be a finite set. For any $X\in H^E$, let $\underline{X}:=\{e\in E: X_e\neq 0\}$ denote the {\em support} of $X$. A {\em  left $H$-matroid on $E$} is a pair $(E, \mathcal{C})$, where $\mathcal{C}\subseteq H^E$ satisfies the following {\em circuit axioms}.
\begin{itemize}
\item[(C0)] $0\not\in \mathcal{C}$.
\item[(C1)] if $X\in \mathcal{C}$ and $\alpha\in H^\star$, then $\alpha\x X\in \mathcal{C}$.
\item[(C2)] if $X,Y\in \mathcal{C}$ and $\underline{X}\subseteq \underline{Y}$, then there exists an $\alpha\in H^\star$ so that $Y = \alpha\x X$.
\item[(C3)] if $X,Y\in\mathcal{C}$ are a modular pair in $\mathcal{C}$ and $e\in E$ is such that $X_e=-Y_e\neq 0$, then there exists a $Z\in\mathcal{C}$ so that $Z_e=0$ and $Z\in X\+ Y$.
\end{itemize}
 In (C3), a pair  $X,Y\in\mathcal{C}$ is {\em modular} if $\underline{X},\underline{Y}$ are modular in $\underline{\mathcal{C}}:=\{\underline{X}: X\in \mathcal{C}\}$, in the sense that there are no two distinct  elements $X'$ and $Y'$ of $\mathcal{C}$ with $\underline X' \cup \underline Y'$ a proper subset of $\underline X \cup \underline Y$.
A {\em  right $H$-matroid} is defined analogously, with $\alpha\x X$ replaced by $ X\x \alpha$ in (C1) and (C2). There is no difference between a left- and a right $H$-matroid if $H$ is commutative, and then we speak of $H$-matroids.

If $M=(E, \mathcal{C})$ is a left $H$-matroid, then $\underline{\mathcal{C}}$ is the set of circuits of a matroid in the traditional sense, the matroid $\underline{M}$ {\em underlying }$M$. If $H$ is the Krasner hyperfield, then $\underline{M}$ determines $M$.  

If $N$ is a matroid on $E$ and $H$ is a skew hyperfield, then a collection $\mathcal{C}\subseteq H^E$ is a {\em left $H$-signature} of $N$ if $\mathcal{C}$ satisfies (C0), (C1), and (C2), and $\underline{\mathcal{C}}$ is the collection of circuits of $N$. Then $M=(E, \mathcal{C})$ is a left $H$-matroid by definition if and only if $\mathcal{C}$ satisfies (C3). 

If $X, Y\in H^E$, then we say that $X$ {\em is  orthogonal to }$Y$, notation $X\perp Y$, if 
$0\in X\cdot Y:=\bigboxplus_e X_e Y_e$.
Sets $\mathcal{C},\mathcal{D}\subseteq H^E$ are {\em $k$-orthogonal}, written $\mathcal{C}\perp_k\mathcal{D}$, if $X\perp Y$ for all $X\in \mathcal{C}$ and $Y\in \mathcal{D}$ such that $|\underline{X}\cap \underline{Y}|\leq k$, and they are simply {\em orthogonal}, written $\mathcal{C}\perp\mathcal{D}$, if $X\perp Y$ for all $X\in \mathcal{C}$ and $Y\in \mathcal{D}$.

Orthogonality gives an alternative way to characterize if a circuit signature determines a matroid. 
\begin{theorem}\label{thm:duality} Let $N$ be a matroid on $E$, let $H$ be a skew hyperfield, an let $\mathcal{C}$ be a left $H$-signature of $N$. Then $M=(E,\mathcal{C})$ is a left $H$-matroid if and only if there exist a right $H$-signature $\mathcal{D}$ of $N^*$ so that $\mathcal{C}\perp_3\mathcal{D}$.
\end{theorem}
If $M=(E,\mathcal{C})$ is a left $H$-matroid, then there is exactly one set $\mathcal{D}$ as in the theorem, and then $M^*:=(E, \mathcal{D})$ is a right $H$-matroid, the {\em dual} of $M$. We say that $M$ has {\em strong duality} if $\mathcal{C}\perp\mathcal{D}$. 

The circuits $\mathcal{D}$ of the dual of $M$ are the {\em cocircuits} of $M$. We may write $\mathcal{C}(M)$ and $\mathcal{D}(M)$ for the circuits and cocircuits of $M$.

In one direction of Theorem \ref{thm:duality} we may drop the assumptions that $\mathcal{C}$ and $\mathcal{D}$ satisfy (C0) and (C2). More precisely,  let's say that a collection $\mathcal{C}\subseteq H^E$ is a {\em weak left $H$-signature} of $N$ if $\mathcal{C}$ satisfies (C1) and $\underline{\mathcal{C}}$ is the collection of circuits of $N$. 

\begin{lemma}\label{lem:weaksigs}
Let $N$ be a matroid on $E$ and $H$ a skew hyperfield. Let $\mathcal{C}$ be a weak left $H$-signature of $N$ and $\mathcal{D}$ a weak right $H$-signature of $N^*$. If $\mathcal{C} \perp_2 \mathcal{D}$ then $\mathcal{C}$ is a left $H$-signature. If $\mathcal{C} \perp_3 \mathcal{D}$ then $(E, \mathcal{C})$ is a left $H$-matroid. 
\end{lemma}
\begin{proof}
Suppose first that $\mathcal{C} \perp_2 \mathcal{D}$. Now $\mathcal{C}$ satisfies (C0) since no circuit of $N$ is empty. To show that it satisfies (C2), suppose that we have $X$ and $Y$ in $\mathcal{C}$ with $\underline{X} \subseteq \underline{Y}$. Since both $\underline{X}$ and $\underline{Y}$ are circuits of $N$, they must be equal. Let $e_0$ be any element of $\underline{X}$ and let $\alpha := Y(e_0) \x X(e_0)^{-1}$. Let $e$ be any other element of $\underline X$. Then there is some cocircuit $D$ of $N$ with $\underline X \cap D = \{e_0, e\}$. Let $Z \in \mathcal{D}$ with $\underline{Z} = D$. Then $0 \in X(e_0) \x Z(e_0) \boxplus X(e) \x Z(e)$, so that $X(e_0) \x Z(e_0) = - X(e) \x Z(e)$. Similarly $Y(e_0) \x Z(e_0) = -Y(e) \x Z(e)$. Then $$\alpha \x X(e) = Y(e_0) \x X(e_0)^{-1} \x X(e) = -Y(e_0) \x Z(e_0) \x Z(e)^{-1} = Y(e)\,.$$
Since $e$ was arbitrary, we have $Y = \alpha \x X$, completing the proof of (C2). Thus $\mathcal{C}$ is a left $H$-signature. 

A dual argument shows that $\mathcal{D}$ is a right $H$-signature. So if $\mathcal{C} \perp_3 \mathcal{D}$ then $(E, \mathcal{C})$ is a left $H$-matroid by Theorem \ref{thm:duality}.
\end{proof}

If $\mu: H\rightarrow H'$ is a homomorphism and $X\in H^E$, then we denote $$\mu_* X:=(\mu(X_e): e\in E).$$  If $M$ is a left $H$-matroid on $E$, and  $\mu_*\mathcal{C}:=\{\mu_*X: X\in \mathcal{C}\}$ then  $\mu_* M:=(E, \mu_*\mathcal{C})$ is a left $H'$-matroid.

\subsection{Rescaling}
If $H$ is a skew hyperfield, $X\in H^E$ and $\rho: E\rightarrow H^\star$, then {\em right rescaling $X$ by $\rho$} yields the vector $X\rho\in H^E$ with entries $(X\rho)_e=X_e\rho(e)$. Similarly, {\em left rescaling} gives a vector $\rho X$. The function $\rho$ is called a {\em scaling vector} in this context, and we use the shorthand $\rho^{-1}$ for the function $E\rightarrow H^\star$ such that $\rho^{-1}(e)= \rho(e)^{-1}$. 

If $X,Y\in H^E$ and $\rho:E\rightarrow H^\star$ is a rescaling vector, then clearly $X\perp Y$ if and only if $(X\rho^{-1})\perp (\rho Y)$.  
By extension, we have  $\mathcal{C} \perp_k \mathcal{D}$ if and only if  $(\mathcal{C} \rho^{-1})\perp_k (\rho\mathcal{D})$ for any sets $\mathcal{C}, \mathcal{D}\subseteq H^E$, where we wrote
$$\mathcal{C}\rho^{-1}:=\{X\rho^{-1}: X\in \mathcal{C}\}\text{ and } \rho \mathcal{D}:=\{\rho Y: Y\in \mathcal{D}\}.$$
Hence if $M$ is a left $H$-matroid on $E$ with circuits $\mathcal{C}$ and cocircuits $\mathcal{D}$, then $\mathcal{C}\rho^{-1}$ and $\rho \mathcal{D}$ are the circuits and cocircuits of a left $H$-matroid $M^\rho$. We say that $M^\rho$ arises from $M$ by {\em rescaling}.

\subsection{Minors}
Let $H$ be a skew hyperfield and let $E$ be a finite set. For any $\mathcal{Q}\subseteq H^E$, we define 
$$\minsupp(\mathcal{Q}):=\{Q\in \mathcal{Q}: \underline{R}\subseteq \underline{Q} \Rightarrow\underline{R}= \underline{Q}\text{ for all } R\in \mathcal{Q}\}.$$
Since $\minsupp(\mathcal{Q})\subseteq \mathcal{Q}$, it is evident that e.g. $\mathcal{Q}\perp_k\mathcal{R}\Rightarrow\minsupp (Q)\perp_k \mathcal{R}$.
For a  $Q\in H^E$ and an $e\in E$, we write $Q\del e$ for the restriction of  $Q$ to $E\del e$. For any $e\in E$, we write $$\mathcal{Q}_e:=\{Q\del e: Q\in \mathcal{Q}, Q_e=0\}\text{ and }\mathcal{Q}^e:=\{Q\del e: Q\in \mathcal{Q}\}.$$ Clearly, 
$\mathcal{Q}\perp_k\mathcal{R}\Longrightarrow \mathcal{Q}_e\perp_k\mathcal{R}^e.$

Let  $M$ be a left $H$-matroid on $E$ with circuits $\mathcal{C}$ and cocircuits $\mathcal{D}$. Since $\mathcal{C}\perp_3\mathcal{D}$, we have 
$$\mathcal{C}_e\perp_3 \minsupp(\mathcal{D}^e \del \{0\})\text{ and } \minsupp(\mathcal{C}^e \del \{0\})\perp_3\mathcal{D}_e$$
for any $e\in E$. By Theorem \ref{thm:duality} $M\del e:=(E\del e, \mathcal{C}_e)$ and $M/e:=(E\del e, \minsupp(\mathcal{C}^e \del \{0\}))$ are both left $H$-matroids. We say that $M\del e$ arises from $M$ by {\em deleting $e$} and $M/ e$ arises from $M$ by {\em contracting $e$}. 
Evidently, $M^*\del e=(M/e)^*$ and $M^*/e=(M\del e)^*$.

A matroid $M'$ is a {\em minor} of $M$ if $M'$ is be obtained by deleting and contracting any number of elements of $M$. If $M$ has strong duality, then this property is inherited by $M\del e$ and $M/e$, and hence by all minors of $M$. The standard notations for matroid minors extend to matroids over skew hyperfields in the obvious way, so that for example for $X \subseteq E$ the restriction $M|X$ of $M$ to $X$ is obtained from $M$ by deleting all elements of $E \del X$.

\subsection{Vectors, covectors, and perfection}
A {\em vector} of a left $H$-matroid $M$ is any $V\in H^E$ so that $V\perp Y$ for all cocircuits $Y\in \mathcal{D}$, and a {\em covector} is a $U\in H^E$ so that $X\perp U$ for all circuits $X\in\mathcal{C}$. We  write $\mathcal{V}(M), \mathcal{U}(M)$ for the sets of  vectors and covectors of $M$.

\begin{lemma}\label{lem:strong_ctov} Let $M$ be a left $H$-matroid with strong duality. Then  
$$\mathcal{C}(M)= \minsupp(\mathcal{V}(M)\del \{0\})\text{ and }\mathcal{D}(M)= \minsupp(\mathcal{U}(M)\del \{0\}).$$ 
\end{lemma}

We say that a matroid $M$ is {\em perfect} if $\mathcal{V}(M)\perp\mathcal{U}(M)$, and that a hyperfield $H$ is {\em perfect} if each matroid $M$ over $H$ is perfect. Not all hyperfields are perfect.

\begin{theorem} \label{thm:Kperfect} The Krasner hyperfield, the sign hyperfield,  and skew fields are perfect.\end{theorem}

\begin{theorem}\label{lem:perfect_strong} Let $H$ be a perfect hyperfield. Then any left $H$-matroid has strong duality.
\end{theorem}
\begin{proof}
It clearly suffices to show for all natural numbers $k$ that for any left $H$-matroid $M$ we have $\mathcal{C}(M) \perp_k \mathcal{D}(M)$. We shall do this by induction on $k$.

The cases with $k \leq 3$ are clear, so suppose $k > 3$. We must show that for any $X \in \mathcal{C}(M)$ and $Y \in \mathcal{D}(M)$ with $|\underline X \cap \underline Y| = k$ we have $X \perp Y$. 

Let $F = \underline X \cap \underline Y$. Let $x$, $y$ be 2 distinct elements of $F$. Let $B$ be a base of $\underline M$ including $(\underline X \del F) \cup \{x,y\}$ and disjoint from $\underline Y \del \{x, y\}$. Let $M' = (M/(B \del \{x,y\}))|F$. Thus $\{x,y\}$ is a base of $\underline{M'}$, which therefore has rank 2. Since the ground set $F$ of $M'$ has size $k \geq 4$, $\underline{M'}$ cannot consist of a single circuit or a single cocircuit.

Let $X'$ and $Y'$ be the restrictions of $X$ and $Y$ to $F$. For any circuit $Z'$ of $M'$ there is a circuit $Z$ of $M$ with $\underline Z \subseteq \underline Z' \cup (B \del \{x,y\})$. Then $|\underline Z \cap \underline Y| = |\underline Z'| < k$ and so by the induction hypothesis $Z \perp Y$, so that $Z' \perp Y'$. Thus $Y'$ is a covector of $M'$. A dual argument shows that $X'$ is a vector of $M'$. Since $H$ is perfect we have $X' \perp Y'$ and so $X \perp Y$.
\end{proof}

Rescaling of a matroid $M$ has a straightforward effect on the vectors and covectors. 
\begin{lemma} If $M$ is a left $H$ matroid on $E$ and $\rho:E\rightarrow H^\star$, then $\mathcal{V}(M^\rho)=\mathcal{V}(M)\rho$ and $\mathcal{U}(M^\rho)=\rho\mathcal{U}(M)$
\end{lemma}

Matroids, oriented matroids, and linear spaces exhibit the following natural relation between the vectors of a matroid $M$ and its direct minors. 
\begin{theorem} Let $H$ be the Krasner hyperfield, the sign hyperfield, or a skew field, let $M$ be a $H$-matroid on $E$ and let $e\in E$.   Then $\mathcal{V}(M/ e) = \mathcal{V}(M)^e$ and $\mathcal{V}(M\del  e) = \mathcal{V}(M)_e$.
\end{theorem}
For other hyperfields $H$ this statement may fail, but the following holds in general.
\begin{theorem}[Anderson \cite{Anderson2019}] \label{thm:half_ext}Let $H$ be a skew hyperfield, let $M$ be matroid over $H$ on $E$, and let $e\in E$. Then $\mathcal{V}(M/e) \supseteq \mathcal{V}(M)^e$ and $\mathcal{V}(M\del e) \supseteq \mathcal{V}(M)_e$.
 \end{theorem}
 \proof Suppose first that $M$ is a left $H$-matroid. 

 $\mathcal{V}(M/e) \supseteq \mathcal{V}(M)^e$: Let $W\in \mathcal{V}(M)^e$, and let $V\in \mathcal{V}(M)$ be such that $V\del e= W$. If $W\not \in \mathcal{V}(M/e)$, then there is a cocircuit $Z\in \mathcal{D}(M/e)$ so that $W\not\perp Z$. Then for the cocircuit $Y\in \mathcal{D}(M)$ so that $Y\del e=Z$ and $Y_e=0$, we have $V\not\perp Y$, a contradiction.
 
 $\mathcal{V}(M\del e) \supseteq \mathcal{V}(M)_e$: Let $W\in \mathcal{V}(M)_e$, and let $V\in \mathcal{V}(M)$ be such that $V\del e= W$ and $V_e=0$. If $W\not \in \mathcal{V}(M/e)$, then there is a cocircuit $Z\in \mathcal{D}(M\del e)$ so that $W\not\perp Z$. Then for the cocircuit $Y\in \mathcal{D}(M)$ so that $Y\del e=Z$, we have $V\not\perp Y$ since $V_e=0$, a contradiction.

A straightforward adaptation of this argument  settles the case when $M$ is a right $H$-matroid. 
\endproof

\subsection*{Example: Oriented matroids} Matroids over the hyperfield of signs $\mathbb{S}$ are exactly oriented matroids (see \cite[Thm. 3.6.1]{OM}), and the above defininitions of circuit, cocircuit, vector and covector generalize the oriented matroid definitions. All oriented matroids are perfect. 
For the hyperfield of signs,  there is a single-valued composition $\circ: \mathbb{S}\times \mathbb{S}\rightarrow \mathbb{S}$ defined by 
$$a\circ b:=\left\{\begin{array}{ll} a&\text{if }a\neq 0\\ b&\text{otherwise}\end{array}\right.$$
The vector axiomatization of oriented matroids can be stated as follows (cf. \cite[Thm. 3.7.9]{OM}).
\begin{theorem}\label{thm:sign_axiom}Let $E$ be a finite set, and let $\mathcal{V}\subseteq \mathbb{S}^E$. There is an $\mathbb{S}$-matroid $M$ such that $\mathcal{V}=\mathcal{V}(M)$ if and only if
\begin{enumerate}
\item[(V0)] $0\in \mathcal{V}$.
\item[(V1)] if $a\in \mathbb{S}$ and $V\in \mathcal{V}$, then  $aV\in \mathcal{V}$.
\item[(V2)] if $V, W\in \mathcal{V}$, then $V\circ W\in \mathcal{V}$. 
\item[(V3)] if $V, W\in \mathcal{V}, e\in E$ such that $V_e=-W_e\neq 0$, then there is a $Z\in \mathcal{V}$ such that $Z\in V\+ W$ and $Z_e=0$.
\end{enumerate}
Then $\mathcal{C}(M)=\minsupp(\mathcal{V}\del\{0\})$.
\end{theorem}
These vector axioms have no obvious counterpart for matroids over general hyperfields, if only because there is no clear way to define a composition $\circ$ for all hyperfields. 


\section{Valuated matroids}
\subsection{Valuative skew hyperfields}
Each totally ordered group $(\Gamma, \x, <)$ determines a skew hyperfield $\Gamma_{\max}:=(\Gamma\cup\{0\}, \x, \+, 1,0)$ which inherits its multiplication on $\Gamma_{\max}^\star$ from $\Gamma$, and with a hyperaddition given by 
$$x\+ y =\left\{\begin{array}{ll}\max\{x,y\}&\text{if } x\neq y\\ \{z\in \Gamma: z\leq x\}\cup\{0\}&\text{if } x=y\end{array}\right.
$$
for $x,y\in \Gamma_{\max}^\star$. The linear order $<$ of $\Gamma$ extends to $\Gamma_{\max}$ by setting $0<x$ for all $x\in \Gamma$.

If $\Gamma$ is abelian, then a $\Gamma_{\max}$-matroid is exactly a valuated matroid as defined by  Dress and Wenzel \cite{DressWenzel1992b}.

\begin{lemma} \label{lem:valstrong} Let $M$ be a left $\Gamma_{\max}$-matroid on $E$ with circuits $\mathcal{C}$ and cocircuits $\mathcal{D}$. Then $\mathcal{C}\perp \mathcal{D}$.\end{lemma}
\proof We will show that for all left $\Gamma_{\max}$-matroids $M$ on $E$ and all $X\in\mathcal{C}(M)$ and $Y\in \mathcal{D}(M)$, we have $X\perp Y$, by induction on $|E|+|\underline{X}\cap \underline{Y}|$.

If there is an $e\in \underline{Y}\del\underline{X}$, then consider the restrictions $X':=X\del e\in\mathcal{C}(M\del e)$  and $Y':=Y\del e\in \mathcal{D}(M\del e)$. We have $|E(M\del e)|<|E(M)|$, so by induction we obtain $\bigboxplus_e X_e\x Y_e=\bigboxplus_e X'_e\x Y'_e\ni 0$, and hence $X\perp Y$. Hence $\underline{Y}\subseteq \underline{X}$. By the dual argument, we also obtain $\underline{X}\subseteq \underline{Y}$.

By assumption $X\perp Y$ whenever $|\underline{X}\cap\underline{Y}|\leq 3$, so we may assume $|\underline{X}\cap\underline{Y}|> 3$. 
If $X\not \perp Y$, then there is an $e\in \underline{Y}$ so that $X_e\x Y_e> X_f \x Y_f$ for all $f\in \underline{Y}\del \{e\}$. 
Pick any $T\in\mathcal{C}$ such that $X, T$ is a modular pair of circuits. Scaling $T$, we can make sure that $T_f\leq X_f$ for all $f\in \underline{Y}$ and $T_g=X_g$ for some $g\in  \underline{Y}$. 

If $T_e< X_e$, then fix $g\in\underline{Y}$ so that $T_g=X_g$. By modular circuit elimination,  there exists a $Z\in X\+T$ so that $Z_g=0$  and $Z\in X\+ T$. Then $Z_e=X_e$ and $Z_f\leq X_f$ for all $f\in \underline{Z}\cap \underline{Y}$. Hence $Z_e \x Y_e =X_e \x Y_e> X_f\x Y_f\geq  Z_f\x Y_f$ for all $f\in \underline{Z}\cap \underline{Y}\del\{e\}$, so that $Z\not\perp Y$. Since $Z_g=0$, we also have $|\underline{Z}\cap \underline{Y}|< |\underline{X}\cap\underline{Y}|$, and hence by induction $Z\perp Y$, a contradiction.

If $T_e= X_e$, then $T_e\x Y_e=X_e\x Y_e> X_f \x Y_f\geq T_f\x Y_f$ for all $f\in \underline{T}\cap \underline{Y}\del \{e\}$, and hence $T\not\perp Y$. Since $X, T$ are a modular pair,  $\underline{T}$ is distinct from $\underline{X}=\underline{Y}$. Hence $|\underline{T}\cap \underline{Y}|< |\underline{X}\cap\underline{Y}|$ and by induction $T\perp Y$, a contradiction.
\endproof

\subsection{The residue matroid}
For any vector $X\in \Gamma_{\max}^E$, let $X^\tl:=\left\{e\in E: X_e = \max_f X_f\right\}$. The following observation is key to our analysis of matroids over $\Gamma_{\max}$.
\begin{lemma} \label{lem:val_orth}Let $E$ be a finite set and let $X, Y\in \Gamma_{max}^E$. Then $X\perp Y\Rightarrow X^\tl\perp Y^\tl$. Conversely, if $X^\tl\cap Y^\tl\neq \emptyset$, then $X^\tl\perp Y^\tl\Rightarrow X\perp Y$.
\end{lemma}
\proof If $X^\tl\cap Y^\tl=\emptyset$, then $X^\tl\perp Y^\tl$ and we are done. If 
$X^\tl\cap Y^\tl\neq \emptyset$, then $X_eY_e=\max_f X_fY_f$ if and only if $e\in  X^\tl\cap Y^\tl$. Then $X\perp Y$ if and only if  
$$0\in \bigboxplus_e X_e\x Y_e = \bigboxplus_{e\in X^\tl\cap Y^\tl} X_e\x Y_e$$
if and only if $|X^\tl\cap Y^\tl|\neq 1$, that is if $X^\tl\perp Y^\tl$.\endproof
For a set $\mathcal{Q}\subseteq \Gamma_{\max}^E$, we put $\mathcal{Q}_0:=\minsupp\{ Q^\tl: Q\in\mathcal{Q}\}.$  
By the Lemma, we have  $X\perp \mathcal{Q}\Rightarrow X^\tl\perp \mathcal{Q}_0$. We next show that if $\mathcal{C}$ is  the set of circuits of a matroid over $\Gamma_{\max}$, then $\mathcal{C}_0$ is the set of circuits of a matroid. Our argument will make use of  a theorem of Minty.
\begin{theorem}[Minty \cite{Minty}] \label{thm:minty}Let $E$ be a finite set and let $\mathcal{C}$ and $\mathcal{D}$ be sets of nonempty subsets of $E$. Then there is a matroid $M$ on $E$ with circuits $\mathcal{C}$ and cocircuits  $\mathcal{D}$ if and only if:
\begin{itemize}
\item[(M0)] if $C,C'\in\mathcal{C}$ and $C\subseteq C'$, then $C=C'$; if $D,D'\in\mathcal{D}$ and $D\subseteq D'$, then $D=D'$.
\item[(M1)] there are no $C\in \mathcal{C}$ and $D\in \mathcal{D}$ so that $|C\cap D|=1$.
\item[(M2)] for each partition of $E$ in parts $B,G,R$ so that $|G|=1$, either
\begin{itemize}
\item there is a $C\in \mathcal{C}$ such that $G\subseteq C\subseteq R\cup G$; or
\item  there is a $D\in \mathcal{D}$ such that $G\subseteq D\subseteq B\cup G$.
\end{itemize}
\end{itemize}
\end{theorem}

We will need the following consequence of this characterisation, which is also easy to derive from \cite[Theorem 4.1]{Minty}.

\begin{theorem}\label{thm:newminty}
Let $E$ be a finite set and let $\mathcal{C}$ and $\mathcal{D}$ be sets of subsets of $E$ satisfying (M1) and (M2) from Theorem \ref{thm:minty}. Let $\mathcal{C}_0$ be the set of minimal nonempty elements of $\mathcal{C}$ and $\mathcal{D}_0$ be the set of minimal nonempty elements of $\mathcal{D}$.  Then there is a matroid $M$ on $E$ with circuits $\mathcal{C}_0$ and cocircuits $\mathcal{D}_0$. 
\end{theorem}
\begin{proof}
It suffices to show that $\mathcal{C}_0$ and $\mathcal{D}_0$ satisfy (M0)-(M2) from Theorem \ref{thm:minty}. By definition both (M0) and (M1) hold, so it suffices to check (M2). So suppose we have a partition of $E$ in parts $B$, $G$ and $R$ with $|G| = 1$. Using (M2) for $\mathcal{C}$ and $\mathcal{D}$, we may suppose without loss of generality that there is some $C \in \mathcal C$ such that $G \subseteq C \subseteq R \cup G$. Fix a minimal such $C$.

We now show that $C \in \mathcal{C}_0$. Suppose for a contradiction that it is not. Then there must be some $C' \in \mathcal{C}_0$ with $C' \subseteq C$. By the minimality of $C$ we cannot have $G \subseteq C'$. Let $e$ be any element of $C'$. Let $B' := (E \setminus C) \cup \{e\}$ and $R' := C \setminus (G \cup \{e\})$, so that $B'$, $G$ and $R'$ give a partition of $E$ with $|G| = 1$. Now we apply (D2) for $\mathcal{C}$ and $\mathcal{D}$ to this partition. The first possibility is that we obtain some $C'' \in \mathcal{C}$ with $G \subseteq C'' \subseteq R \cup G = C \setminus \{e\}$, but this cannot happen since it would contradict our choice of $C$. The other possibility is that there is some $D \in \mathcal{D}$ with $G \subseteq D \subseteq B' \cup G$. In that case we have $G \subseteq C \cap D \subseteq G \cup \{e\}$, so that by $(M1)$ we have $e \in D$. But then $C' \cap D = \{e\}$, contradicting (M1).

This contradiction shows that $C \in \mathcal{C}_0$, and since $G \subseteq C \subseteq G \cup R$ it witnesses that (M2) holds for the partition of $E$ into $B$, $G$ and $R$.
\end{proof}

\begin{lemma} \label{lem:val_residue}Let $M$ be a left $\Gamma_{\max}$-matroid on $E$ with circuits $\mathcal{C}$ and cocircuits $\mathcal{D}$. Then there exists a matroid $M_0$ on $E$ with circuits $\mathcal{C}_0$ and cocircuits $\mathcal{D}_0$.
\end{lemma}
\proof We prove the theorem by induction on $|E|$. If $|E|\leq 1$, then it is straightforward that  $M_0=\underline{M}$ is as required. If $|E|>1$, we prove that $\mathcal{C}_0$ and $\mathcal{D}_0$ are the circuits and cocircuits of a matroid $M_0$ by applying Theorem \ref{thm:newminty} to the sets $\mathcal{C}_1 := \{X^\tl:X \in \mathcal{C}\}$ and $\mathcal{D}_1 := \{Y^\tl:Y \in \mathcal{C}\}$. So what we must show is that $\mathcal{C}_1$ and $\mathcal{D}_1$ satisfy (M1) and (M2).

To see (M1), let $C\in \mathcal{C}_1$ and $D\in \mathcal{D}_1$, and let $X\in\mathcal{C}$ and $Y\in \mathcal{D}$ be such that $C=X^\tl$ and $D=Y^\tl$.  Since $X\perp Y$, it follows that $C=X^\tl\perp Y^\tl=D$ by Lemma \ref{lem:val_orth}.

Finally, we show (M2). Let  $E$ be partitioned in parts $B,G,R$ so that $|G|=1$. Since $|E|>1$, there is at least one element $e \in E\del G$, so $e\in R$ or $e\in B$. Replacing $M$ with $M^*$ if $e\in B$, may assume that $e\in R$. 

By the induction hypothesis, the statement of the theorem holds for $M\del e$. Applying Minty's Theorem to the matroid $(M\del e)_0$, there exists either
\begin{itemize}
\item  a $C\in (\mathcal{C}_e)_0$ so that $G\subseteq C\subseteq (R\del e) \cup G$, or 
\item a $D\in \minsupp(\mathcal{D}^e\del \{0\})_0$ so that $G\subseteq D\subseteq B \cup G$. 
\end{itemize}
In the former case, there is an $X\in \mathcal{C}_e$ so that $C=X^\tl$, and hence there is an $X'\in \mathcal{C}$ with $X'\del e=X$ and $X'_e=0$, so that $C=X'^\tl\in \mathcal{C}_1$. Then we are done, since $C$ satisfies $G\subseteq C\subseteq R \cup G$.
In the latter case, there exists a $Y\in \mathcal{D}^e$ so that $D=Y^\tl$. Then there exists a $Z\in\mathcal{D}$ so that $Y$ is the restriction of $Z$ to $E-e$. If $Z^\tl=Y^\tl=D$, then we are done, and hence $Z_e\geq \max_{f\neq e} Z_f$. 

By the induction hypothesis, the statement of the theorem holds for $M/ e$. Applying Minty's Theorem to the matroid $(M/ e)_0$, there exists either
\begin{itemize}
\item  a $C\in \minsupp(\mathcal{C}^e)_0$ so that $G\subseteq C\subseteq (R\del e) \cup G$, or 
\item a $D\in (\mathcal{D}_e)_0$ so that $G\subseteq D\subseteq B \cup G$. 
\end{itemize}
In the latter case, there exists a $Y\in \mathcal{D}_e$ so that $D=Y^\tl$, and hence there is a $Y'\in \mathcal{C}$ with $Y'\del e=Y$ and $Y'_e=0$, so that $D=Y'^\tl\in \mathcal{D}_1$.   Since  $D$ satisfies $G\subseteq D\subseteq B \cup G$, we are done. 
In the former case, there is an $X\in \mathcal{C}^e$ so that $C=X^\tl$. Then there exists a $T\in\mathcal{C}$ so that $X$ is the restriction of $T$ to $E-e$. If $T^\tl\subseteq \{e\}\cup X^\tl=\{e\}\cup D$, then we are done, and hence $T_e>\max_{f\neq e} T_f$. 

Summing up, we have obtained a $Z\in\mathcal{D}$ so that $Z_e\geq \max_{f\neq e} Z_f$ from considering $M\del e$ as well as a $T\in\mathcal{C}$ so that $T_e>\max_{f\neq e} T_f$ from  considering $M/e$. It follows that  $T_e \x Z_e>T_f \x Z_f$ for all $f\neq e$, so that $\bigboxplus_f T_f \x Z_f = T_e \x Z_e\not \ni 0$ and hence $T\not\perp Z$. This contradicts Lemma \ref{lem:valstrong}.\endproof 
If $M$ is a left $\Gamma_{\max}$-matroid, then the matroid $M_0$ of Lemma \ref{lem:val_residue} is the {\em residue matroid} of $M$. If one assumes that $\Gamma$ is commutative, then the bases of $M_0$ are exactly the maximizers of the Grassmann-Pl\"ucker coordinates of $M$. Thus in the commutative case, our residue matroid coincides with a construct proposed by Dress and Wenzel for valuated matroids \cite{DressWenzel1992a}, and Lemma \ref{lem:val_residue} generalizes Proposition 2.9(i) of that paper to non-commutative matroid valuations.

Residue matroids are well-behaved with respect to certain minors:
\begin{lemma}\label{lem:resminor}
Let $M$ be a left $\Gamma_{\max}$-matroid on $E$ with circuits $\mathcal{C}$ and let $e \in E$. If $e$ is not a loop of $M_0$ then $(M/e)_0 = M_0/e$. If $e$ is not a coloop of $M_0$ then $(M\backslash e)_0 = M_0 \backslash e$.
\end{lemma}
\begin{proof}
The circuits of $(M/e)_0$ are the minimal nonempty sets of the form $(X \setminus e)^\tl$ and those of $M_0/e$ are the minimal nonempty sets of the form $X^\tl \setminus e$ with $X \in \mathcal{C}$. But if $e$ is not a loop of $M_0$ then for any $X \in \mathcal C$ we have $(X \setminus e)^\tl = X^\tl \setminus e$. The second statement can be proved with a dual argument.
\end{proof}


\begin{lemma}\label{lem:extendintospan}
Let $M$ be a left $\Gamma_{\max}$-matroid on $E$ with circuits $\mathcal{C}$. Let $C$ be any circuit of $M_0$ and $S$ any spanning set of $M_0$. Then there is $X \in \mathcal{C}$ such that $X^\tl = C$ and $\underline{X} \subseteq S \cup C$.
\end{lemma}
\begin{proof}
We repeatedly apply Lemma \ref{lem:resminor} to delete all the elements of $E \setminus (S \cup C)$, giving $M_0 | (S \cup C) = (M | (S \cup C))_0$, from which the statement follows. None of the elements that we delete are coloops, since they are spanned by $S$.
\end{proof}

\begin{lemma} \label{lem:val_space}Let $M$ be a left $\Gamma_{\max}$-matroid. Then $V\in \mathcal{V}(M)\Rightarrow V^\tl\in \mathcal{V}(M_0)$ and $U\in \mathcal{U}(M)\Rightarrow U^\tl\in \mathcal{U}(M_0)$.\end{lemma}
\proof If $V\in \mathcal{V}(M)$, then by definition $V\perp Y$ for all $Y\in \mathcal{D}(M)$, hence $V^\tl\perp Y^\tl$ for all $Y\in \mathcal{D}(M)$.  Then by Lemma \ref{lem:val_orth}, we have $V^\tl\perp D$ for all $D\in\mathcal{D}(M)_0$, so that by definition $V^\tl\in \mathcal{V}(M_0)$. 
The argument for $\mathcal{U}$ is analogous.\endproof

 \begin{theorem} \label{thm:val_perfect} $\Gamma_{\max}$ is perfect. \end{theorem}
\proof  Let $M$ be a left $\Gamma_{\max}$-matroid, let $V\in \mathcal{V}(M)$ and  $U\in \mathcal{U}(M)$. We need to show that $V\perp U$. Let $g\in\Gamma$ be such that  $V_e\x U_e> g \x U_f$ for all $e\in \underline{V}\cap\underline{U}$ and $f\in \underline{U}\del\underline{V}$. Let $\rho:E\rightarrow \Gamma$ be determined by 
$$\rho(e)=\left\{\begin{array}{ll} V_e&\text{if } V_e\neq 0\\ g&\text{otherwise}\end{array}\right.$$
Then $V \rho^{-1}\in \{0,1\}^E$ and $\emptyset\neq\left(\rho U\right)^\tl\subseteq \underline{V}=\left(V \rho^{-1}\right)^\tl$, so that  $\left(V \rho^{-1}\right)^\tl\cap (\rho U)^\tl\neq \emptyset$. Since $V\rho^{-1}\perp \rho U$ if and only if $V\perp U$, we may assume that $\rho\equiv 1$ by replacing $M$ with $M^\rho$ if necessary. Then $V^\tl\cap U^\tl\neq \emptyset$. We have $V^\tl\in \mathcal{V}(M_0)$, $U^\tl\in \mathcal{U}(M_0)$ by Lemma \ref{lem:val_space}, and $\mathcal{V}(M_0)\perp\mathcal{U}(M_0)$ since $M_0$ is an ordinary matroid and the Krasner hyperfield is perfect. By Lemma \ref{lem:val_orth}, we have $V\perp U$, as required.\endproof

\subsection{Vector axioms}
 For $X, Y\in \Gamma_{\max}^E$, let $X\circ Y\in \Gamma_{\max}^E$ be the vector so that $(X\circ Y)_e=\max\{X_e, Y_e\}$ for all $e\in E$. Clearly $(X\circ Y)\circ Z=X\circ(Y\circ Z)$, and we will omit parenthesis in such expressions in what follows.
 \begin{lemma} \label{lem:val_closed}Let $M$ be a left $\Gamma_{\max}$-matroid and let $V, W\in \mathcal{V}(M)$. Then $V\circ W\in \mathcal{V}(M)$.\end{lemma}
 \proof Let $V, W\in \mathcal{V}(M)$. If $V\circ W\not\in \mathcal{V}(M)$, then there is a $Y\in \mathcal{D}(M)$ so that $(V\circ W)\not\perp Y$. Then there is an $e$ so that 
 $(V\circ W)_e Y_e>  (V\circ W)_f Y_f$ for all $f\in E\del e$. Without loss of generality, we have $(V\circ W)_e=V_e$, so that 
 $$V_eY_e = (V\circ W)_e Y_e>  (V\circ W)_f Y_f\geq V_fY_f$$ for all $f\in E\del e$. It follows that $V\not\perp Y$, contradicting that $V\in \mathcal{V}(M)$.\endproof
 
\begin{theorem} \label{thm:val_inner}Let $M$ be a left $\Gamma_{\max}$-matroid. Then
$ \mathcal{V}(M)=\{X^1\circ\cdots\circ X^k: X^i\in \mathcal{C}(M), k\leq r^*(M)\}$.
\end{theorem}
\proof Since $\Gamma_{\max}$ is perfect by Theorem \ref{thm:val_perfect}, $M$ has strong duality by Lemma $\ref{lem:perfect_strong}$. So if $X^1,\ldots, X^k\in \mathcal{C}(M)$, then $X^1,\ldots, X^k\in\mathcal{V}(M)$ by Lemma \ref{lem:strong_ctov}, and hence $X^1\circ\cdots\circ X^k\in \mathcal{V}(M)$ by Lemma \ref{lem:val_closed}. 

Conversely, consider a $V\in \mathcal{V}(M)$. If $V=0$ then $V$ is a composition of $k=0$ circuits. Otherwise, we show by induction on $|E|$ that $V= X^1\circ\cdots\circ X^k$ for some $X^i\in \mathcal{C}(M)$ and $k\leq r^*(M)$. If $\underline{V}\neq E$, pick an $e\in E\del \underline{V}$. 
 Then $V\del e\in \mathcal{V}(M\del e)$ by Theorem \ref{thm:half_ext}, and by induction $V\del e = T^1\circ\cdots\circ T^k$ for some $T^i\in \mathcal{C}(M\del e)= \mathcal{C}(M)_e$, with $k\leq r^*(M\del e)\leq r^*(M)$. 
 Taking $X^i\in \mathcal{C}(M)$ so that $X^i_e=0$ and $X^i\del e = T^i$, we obtain $V=X^1\circ\cdots\circ X^k$ as required. 
 Hence we may assume that $\underline{V}=E$.
 Rescaling, we may assume that $V_e=1$ for all $e\in E$. Then $E=V^\tl\in \mathcal{V}(M_0)$ by Lemma \ref{lem:val_space}, and hence there are circuits $C_1,\ldots, C_k$ of $M_0$ so that $V^\tl=E=\bigcup_i C_i$, with $k\leq r^*(M_0)= r^*(M)$. 
Let $X^1,\ldots, X^k$ be the collection of circuits of $M$ so that  $\max_f X^i_f=1$ and $C_i=(X^i)^\tl$ for $i=1,\ldots, k$. 
The vector $Z=X^1\circ\cdots\circ X^k$ clearly has $\max_f Z_f=1$ and hence $Z^\tl=\bigcup_i C_i=E$. It follows that $V=Z=X^1\circ\cdots\circ X^k$, as required.\endproof

 \begin{theorem} \label{thm:val_ext} Let $M$ be a left $\Gamma_{\max}$-matroid on $E$, let $e\in E$. Then $\mathcal{V}(M/e) = \mathcal{V}(M)^e$ and $\mathcal{V}(M\del e) = \mathcal{V}(M)_e$.
\end{theorem}
 \proof$\mathcal{V}(M/e) = \mathcal{V}(M)^e$: By Theorem \ref{thm:half_ext}, it suffices to show that $\mathcal{V}(M/e) \subseteq \mathcal{V}(M)^e$.  Suppose $W\in \mathcal{V}(M/e)$. 
 By Theorem \ref{thm:val_inner} applied to $M/e$, there exist circuits $T^1,\ldots, T^k\in\mathcal{C}(M/e)$ such that $W=T^1\circ\cdots\circ T^k$. Let $X^i\in\mathcal{C}(M)$ be such that $X^i\del e=T^i$, for each $i$. By Lemma \ref{lem:val_closed}, we have $V:= X^1\circ\cdots\circ X^k\in \mathcal{V}(M)$, and moreover $V\del e=T^1\circ\cdots\circ T^k=W$, as required.
 
 $\mathcal{V}(M\del e) = \mathcal{V}(M)_e$: By Theorem \ref{thm:half_ext}, it suffices to show that $\mathcal{V}(M\del e) \subseteq \mathcal{V}(M)_e$.  Suppose $W\in \mathcal{V}(M\del e)$. 
 By Theorem \ref{thm:val_inner} applied to $M\del e$, there exist circuits $T^1,\ldots, T^k\in\mathcal{C}(M\del e)$ such that $W=T^1\circ\cdots\circ T^k$. Let $X^i\in\mathcal{C}(M)$ be such that $X^i\del e=T^i$ and $X^i_e=0$ for each $i$. By Lemma \ref{lem:val_closed}, we have $V:= X^1\circ\cdots\circ X^k\in \mathcal{V}(M)$, and moreover $V\del e=T^1\circ\cdots\circ T^k=W$ and $V_e=X^1_e\circ\cdots\circ X^k_e=0$, as required.
 \endproof

\begin{theorem}\label{thm:val_axiom}Let $E$ be a finite set, and let $\mathcal{V}\subseteq \Gamma_{\max}^E$. There is a left $\Gamma_{\max}$-matroid $M$ such that $\mathcal{V}=\mathcal{V}(M)$ if and only if
\begin{enumerate}
\item[(V0)] $0\in \mathcal{V}$.
\item[(V1)] if $a\in \Gamma$ and $V\in \mathcal{V}$, then  $aV\in \mathcal{V}$.
\item[(V2)] if $V, W\in \mathcal{V}$, then $V\circ W\in \mathcal{V}$. 
\item[(V3)] if $V, W\in \mathcal{V}, e\in E$ such that $V_e=-W_e\neq 0$, then there is a $Z\in \mathcal{V}$ such that $Z\in V\+ W$ and $Z_e=0$.
\end{enumerate}
Then $\mathcal{C}(M)=\minsupp(\mathcal{V}\del\{0\})$.
\end{theorem}

\proof Sufficiency: Suppose $\mathcal{V}$ satisfies (V0),(V1),(V2),(V3). 
Let $\mathcal{C}:=\minsupp(\mathcal{V}\del \{0\})$. 
Then $\mathcal{C}$ satisfies (C1) by (V1). To see (C2), let $X,Y\in\mathcal{C}$ be such that $\underline{X}\subseteq \underline{Y}$. If $Y\neq a X$ for all $a\in \Gamma$, then scaling $X$ so that $Y_e=X_e$ for some $e\in\underline{X}$, we have $X\neq Y$. By (V3), there is a $Z\in \mathcal{V}$ so that $Z_e=0$ and $Z\in X\+ Y$. Then $\emptyset\neq \underline{Z}\subseteq \underline{Y}\del e$, contradicting that $Y\in \mathcal{C}$. We show that $\mathcal{C}$ satisfies the modular circuit elimination axiom (C3).
If $X, Y\in \mathcal{C}$ are a modular pair, and $X_e=Y_e$, then by (V3) there exists a $Z\in \mathcal{V}$ such that $Z\in X\+ Y$ and $Z_e=0$.
 If $Z\not\in \mathcal{C}$, then there exists a $Z'\in\mathcal{C}$ so that $\underline{Z'}$ is a proper subset of $\underline{Z}$. 
 Applying (V3) to $Z, Z'$, $f\in\underline{Z'}\subseteq \underline{Z}$  then  implies the existence of a $Z''\in\mathcal{C}$ such that $\underline{Z''}\subseteq \underline{Z}\del f$. 
 Then the existence of $Z', Z''\in\mathcal{C}$ would contradict the modularity of the pair $X,Y$ in $\mathcal{C}$, since $\underline{Z'}\cup\underline{Z''}\subseteq \underline{Z}\subseteq \underline{X}\cup\underline{Y}\del\{e\}$. Hence, we have $Z\in \mathcal{C}$. 
 This proves that $\mathcal{C}$ also satisfies modular circuit elimination, so that $\mathcal{C}=\mathcal{C}(M)$ for some left $\Gamma_{\max}$-matroid $M$. 
 We show that $\mathcal{V}=\mathcal{V}(M)$. We have $\mathcal{V}(M)=\{X^1\circ\cdots\circ X^k: X^i\in \mathcal{C}(M), k\leq r^*(M)\}$ by Theorem \ref{thm:val_inner}. Since $\mathcal{V}\supseteq \mathcal{C}=\mathcal{C}(M)$, and $\mathcal{V}$ is closed under $\circ$ by (V2), we have  $\mathcal{V}\supseteq\mathcal{V}(M)$. 
 To show $\mathcal{V}\subseteq\mathcal{V}(M)$, suppose $V\in \mathcal{V}\del\mathcal{V}(M)$ and $V$ has minimal support among all such vectors. Let $X\in \mathcal{C}$ be any vector with $\underline{X}\subseteq\underline{V}$. 
 Scale $X$ so that $X\leq V$, with $X_e=V_e$ for some $e$. Then applying (V3) to $V,X,e$ yields a vector $Z$ such that $Z\in V\+ X$. Then $V=X\circ Z$, since if $V_f=0$ then $X_f=0$ and hence $Z_f=0$, and if $V_f>X_f$, then $V_f=Z_f$.
 We have $X\in \mathcal{V}(M)$ as $X\in \mathcal{C}(M)$ and $Z\in \mathcal{V}(M)$ by minimality of $V$. Hence $V\in \mathcal{V}(M)$ by Lemma \ref{lem:val_closed}, contradicting the choice of $V$.

Necessity: If $\mathcal{V}=\mathcal{V}(M)$, then (V0),(V1) are clear, and (V2) is Lemma \ref{lem:val_closed}. We show (V3) for the set of vectors $\mathcal{V}$ of any left $\Gamma_{\max}$-matroid $M$ on $E$, by induction on $|E|$. Let $V,W,e$ be such that $V, W\in \mathcal{V}$, $V_e=W_e\neq 0$ and $V\neq W$. 
Consider the vector $Z\in\Gamma_{\max}^E$ such that $Z_e=0$ and $Z\del e = (V\circ W)\del e$. 
If $Z\in\mathcal{V}$ then we are done, so assume $Z\not\in\mathcal{V}$. 
Then there is a $Y\in\mathcal{D}(M)$ such that $Z\not\perp Y$, so there is an $f\in E$ such that $Z_fY_f>Z_gY_g$ for all $g\in E\del f$. 

We have $Z_f=\max\{V_f, W_f\}$. Interchanging $V,W$ if necessary, we may assume that $Z_f=V_f$. Then $V_fY_f=Z_fY_f>Z_gY_g\geq V_gY_g$ for all $g\in E\del \{e,f\}$, and since $V\perp Y$ it follows that $V_fY_f=V_eY_e$. Since $W_e=V_e$, we also have $$W_eY_e= V_eY_e=V_fY_f=Z_fY_f>Z_gY_g\geq W_gY_g$$
for all $g\in E\del \{e,f\}$, and since $W\perp Y$ it follows that $W_eY_e=W_fY_f\neq 0$. 

Consider the matroid $M':=M/f$. We have $V':=V\del f\in\mathcal{V}(M')$, $W':=W\del f\in\mathcal{V}(M')$, and by our induction hypothesis there is a vector $Z'\in \mathcal{V}(M')$ so that $Z'_e=0$ and $Z'\in V'\+ W'$. Since $\mathcal{V}(M')=V(M)^e$, there is a vector $Z''\in \mathcal{V}(M)$ such that $Z''\del e=Z'$. If $Z''\in V\+W$ then we are done, so we have $Z''_f>Z_f$. Then 
$$Z''_fY_f>Z_fY_f>Z_gY_g\geq  Z''_gY_g$$
for all $g\in E\del \{e,f\}$, and $Z''_eY_e=0$. It follows that $Z''\not\perp Y$, which contradicts that $Z''\in \mathcal{V}(M)$.
\endproof

Theorem \ref{thm:val_axiom}  generalizes the vector axioms for (abelian-)valuated matroids of Murota and Tamura \cite[Theorems 3.4, 3.5, 3.6]{Murota2001}. The same paper contains a characterization of valuated matroids by non-modular circuit axioms, which is extended below.

\begin{theorem}\label{thm:val_circ}Let $E$ be finite set and let $\mathcal{C}\subseteq  \Gamma_{\max}^E$. Then $M=(E,\mathcal{C})$ is a  left $\Gamma_{\max}$-matroid  if and only if (C0), (C1), (C2) and 
\begin{itemize}
\item [(C3)$'$] for any $X, Y\in \mathcal{C}, e,f\in E$ such that $X_e=Y_e\neq 0$ and $X_f>Y_f$, there is a $Z\in \mathcal{C}$ such that $Z_e=0$, $Z_f=X_f$, and $Z\leq X\circ Y$.
\end{itemize}
\end{theorem}
\proof Necessity: Suppose that $M=(E,\mathcal{C})$ is a  left $\Gamma_{\max}$-matroid. Then (C0), (C1), (C2) hold by definition, and we show (C3)$'$. So assume that $X, Y\in \mathcal{C}, e,f\in E$ are such that $X_e=Y_e\neq 0$, and $X_f>Y_f$. By the vector axiom (V3), there exists a $V\in \mathcal{V}(M)$ such that $V\in X\+ Y$ and $V_e=0$. As $X_f>Y_f$, we have $V_f=X_f$. By Theorem \ref{thm:val_inner}, there exist $Z^1,\ldots, Z^k\in \mathcal{C}$ so that $V=Z^1\circ\cdots\circ Z^k$. Pick $i$ so that $V_f=Z^i_f$ and put $Z:=Z^i$.  Then $Z\in \mathcal{C}$, $Z_e=0$ since $V_e=0$,  $Z_f=V_f=X_f$, and $Z\leq V\leq X\circ Y$, as required.

Sufficiency: Suppose (C0), (C1), (C2), (C3)$'$ hold for $M=(E, \mathcal{C})$. To show that $M$ is a  left $\Gamma_{\max}$-matroid it suffices to show (C3). So let $X,Y\in \mathcal{C}$ be a modular pair of circuits so that $X_e=Y_e$. Pick any $f\in \underline{X}\del\underline{Y}$. By (C3)$'$, there exists a $Z\in \mathcal{C}$ such that $Z_e=0$, $Z_f=X_f$, and $Z\leq X\circ Y$. If $Z\in X\+ Y$ then we are done, so let $g\in E$ be such that $Z_g\not\in X_g\+Y_g$. Then $Z_g<X_g\circ Y_g$ and $X_g\neq Y_g$. If $X_g>Y_g$, then apply (C3)$'$ to $(X,Y, e, g)$ to find a $Z'\in \mathcal{C}$ such that $Z'_e=0$,  $Z'_g=X_g$, and $Z'\leq X\circ Y$. Since $X,Y$ are modular and $\underline{Z}\cup\underline{Z'}\subseteq \underline{X}\cup\underline{Y}\del e$, we have  $\underline{Z}=\underline{Z'}$, and hence $Z=\alpha Z'$ for some $\alpha\in H^\star$ by (C2). Then $Z_g< X_g=Z'_g=\alpha Z_g$, and $Z_f=X_f\geq Z'_f=\alpha Z_f$, a contradiction. If $X_g<Y_g$, we apply (C3)$'$ to $(Y, X, e, g)$ to obtain a $Z'$ with 
$Z_g< Y_g=Z'_g=\alpha Z_g$, and $Z_f=X_f\geq Z'_f=\alpha Z_f$, which again yields a contradiction.\endproof

\section{Matroids over stringent hyperfields} 
\subsection{Stringent hyperfields} A skew hyperfield $H$ is {\em stringent} if $a\neq -b$ implies $|a\+b|=1$ for all $a,b\in H$. 
In a recent paper, Bowler and Su \cite{BowlerSu2019} gave a constructive characterization of  stringent skew hyperfields. We next describe their characterization. Let $R$ be a skew hyperfield with hyperaddition $\+^R$, let $(U, \cdot)$ be a group and let  $(\Gamma,\cdot, <)$ be a (bi-)ordered group. Consider an exact sequence of multiplicative groups
$$1\rightarrow R^\star \xrightarrow{\phi}{}  U\xrightarrow{\psi}{} \Gamma\rightarrow 1$$
where $\phi$ is the identity map. Assume that this exact sequence {\em has stable sums}, that is,  the map $r\mapsto u^{-1}ru$ is an automorphism of the hyperfield $R$ for each $u\in U$.

Define  a multiplication $\cdot$ on $U\cup\{0\}$ by extending the multiplication of the group $U$ with $0\cdot x=x\cdot 0=0$, and define a hyperoperation $\+$ on $U\cup\{0\}$ by setting 
$$x\+ y=\left\{\begin{array}{ll} 
\{x\}&\text{if } \psi(x)>\psi(y)\\
\{y\}&\text{if } \psi(x)<\psi(y)\\
(1\+^R yx^{-1})x&\text{if } \psi(x)=\psi(y)\text{ and }0\not\in 1\+^R yx^{-1}\\
(1\+^R yx^{-1})x\cup\{z\in R: \psi(z)<\psi(x)\}&\text{if } \psi(x)=\psi(y)\text{ and }0\in 1\+^R yx^{-1} \\
\end{array}\right.$$
for all $x,y\in U$, and $x\+0=0\+x=\{x\}$ for all $x\in U\cup\{0\}$.
Let $R \rtimes_{U, \psi} \Gamma:=(U\cup\{0\}, \cdot, \+, 1, 0).$  In what follows, whenever we write $R \rtimes_{U, \psi} \Gamma$ we will implicitly assume the above conditions on $R, U, \Gamma, \psi$, in particular that the exact sequence has stable sums. The following are two key results from \cite{BowlerSu2019}.
\begin{lemma}$R \rtimes_{U, \psi} \Gamma$ is a skew hyperfield. If $R$ is stringent, then so is $R \rtimes_{U, \psi} \Gamma$. \end{lemma}
\begin{theorem} \label{thm:BS}Let $H$ be a stringent skew hyperfield. Then $H$ is of the form $R \rtimes_{U, \psi} \Gamma$, where $R$ is either the Krasner or sign hyperfield or a skew field.
\end{theorem}
Stringent hyperfields may arise, for example, from a construction due to Krasner \cite{Krasner1983}.
\begin{theorem}[Krasner,1983] \label{thm:krasner}Let $R$ be a ring and let $G$ be a normal subgroup of $R^\star$. Let 
$$R/G:=(\{rG: r \in R\}, \oplus, \odot, 0G, 1G)$$
where $rG\oplus sG:=\{tG: tG\subseteq rG+sG\}$ and $rG\odot sG:=(rs)G$.
Then $R/G$ is a hyperring and $$R\rightarrow R/G: r\mapsto rG$$ is a hyperring homomorphism. Moreover,  if $R$ is a skew field then $R/G$ is a skew hyperfield.
\end{theorem}
Krasner used this construction to derive hyperfields from valued fields, and we note that some hyperfields that arise this way are stringent.
\begin{lemma} \label{lem:valfield}Let $K$ be a field with valuation $|.|:K\rightarrow \Gamma_{\max}$, and let $G:=
\{1+k: |k|<1\}$. Then $K/G$ is a stringent hyperfield.\end{lemma}
\proof That $G$ is a normal subgroup of $K^*$ was established by Krasner \cite{Krasner1983}.
Hence $K/G$ is a hyperfield by Theorem \ref{thm:krasner}. To see that $K/G$ is stringent, consider two elements $xG, yG$ of $K/G$. If $|x|>|y|$, then $zG\subseteq  xG+yG$ implies  $z= x(1+k)+y(1+k')=x+y+xk+yk'$
where $|k|, |k'|<1$, so that $z=x(1+k'')$ with $|k''|<1$. It follows that $zG\subseteq xG$, so that $xG\oplus yG=\{xG\}$. Similarly, if $|x|<|y|$ then $xG\oplus yG=\{yG\}$. If $|x|=|y|$ and $x+y\neq 0$, then $zG\subseteq  xG+yG$ implies  $z= x(1+k)+y(1+k')=(x+y)(1+k'')$ with $|k''|<1$, so that $xG\oplus yG=\{(x+y)G\}$. In the remaining case $x=-y$, and hence $xG=-yG$, as required.\endproof
In the context of Lemma \ref{lem:valfield}, we can write $K/G=R\rtimes_{U, \psi} \Gamma$. Then $R$ coincides with the {\em residue field} of the valued field $K$ in the usual sense. In general, we will refer to the hyperfield $R$ as the {\em residue} of a stringent hyperfield $H=R\rtimes_{U, \psi} \Gamma$.

If $R$ is the Krasner hyperfield, then $\psi$ is an isomorphism and $R \rtimes_{U, \psi} \Gamma= \Gamma_{\max}$. In this section, we will generalize the results of the previous section on matroids over  $\Gamma_{\max}$  to matroids over stringent hyperfields.

\subsection{The residue matroid}
We next extend the notation which we introduced for valuative hyperfields $\Gamma_{\max}= \mathbb{K}\rtimes_{U,\psi} \Gamma$ to more general hyperfields of the form $H:=R \rtimes_{U, \psi} \Gamma$.
For any $Q\in H^E$, define $Q^\tl\in H^E$ by 
$$Q^\tl_e=\left\{\begin{array}{ll}Q_e&\text{if }|Q_e|=\max_f |Q_f|\\ 0&\text{otherwise}\end{array}\right.$$
For a set $\mathcal{Q}\subseteq H^E$ we put $\mathcal{Q}_0:=\minsupp\{ Q^\tl: Q\in\mathcal{Q}\}\cap R^E$.

\begin{lemma} \label{lem:strat_orth}Let $E$ be a finite set and let $X, Y\in H^E$. Then $X\perp Y\Rightarrow X^\tl\perp Y^\tl$. Conversely, if $\underline{X^\tl}\cap \underline{Y^\tl}\neq \emptyset$, then $X^\tl\perp Y^\tl\Rightarrow X\perp Y$.
\end{lemma}

If $\nu: H\rightarrow H'$ is a a hyperfield homomorphism and $M$ is a left or right matroid over $H$, then we would ordinarily denote the push-forward as $\nu_* M$. For the valuation $|.|:H\rightarrow \Gamma_{\max}$, we write $|M|:=|.|_*M$.

\begin{lemma}  \label{lem:H_residue} Let $H=R \rtimes_{U, \psi} \Gamma$, and let let $M$ be a left $H$-matroid on $E$ with circuits $\mathcal{C}$ and cocircuits $\mathcal{D}$. There exists a left $R$-matroid $M_0$ on $E$ with circuits $\mathcal{C}_0$ and cocircuits $\mathcal{D}_0$, and we have $\underline{M_0}=|M|_0$.
\end{lemma}
	
\begin{proof}
By induction on $|E|$. The case $|E| = 0$ is trivial, so we may suppose that $|E| \geq 1$.

We begin by showing that $\mathcal{C}_0 \perp_2 \mathcal{D}_0$. Suppose that $X \in \mathcal{C}_0$ and $Y \in \mathcal{D}_0$ with $|\underline{X} \cap \underline{Y}| \leq 2$. If $|\underline{X} \cap \underline{Y}| = 0$ then clearly $X_0 \perp Y_0$, and we cannot have $|\underline{X} \cap \underline{Y}| = 1$ since $|M|_0$ is a matroid. So we may assume that $|\underline{X} \cap \underline{Y}| = 2$. Call its two elements $x$ and $y$. 

Let $\hat X \in \mathcal{C}$ and $\hat Y \in \mathcal{D}$ be such that $\hat X^\tl = X$ and $\hat Y^\tl = Y$. If $|\underline{\hat X} \cap \underline{\hat Y}| \leq 3$ then $\hat X \perp \hat Y$ and so $X \perp Y$ by Lemma \ref{lem:strat_orth}. So we may assume that $|\underline{\hat X} \cap \underline{\hat Y}| \geq 4$. Let $z$ and $t$ be distinct elements of $(\underline{\hat X} \cap \underline{\hat Y}) \setminus x \setminus y$.

First we consider the case that $\underline{\hat X} \neq \underline{\hat Y}$. In this case, without loss of generality there is some $e \in \underline{\hat X} \setminus \underline{\hat Y}$. Applying the induction hypothesis to $M/e$ yields the desired result. 

Now consider the case that $\underline{\hat X} = \underline{\hat Y}$. Let $k = |\underline {\hat X}|$. Then the rank and corank of $\underline M$ are both at least $k-1$, so $M$ has at least $k-2 \geq 2$ elements outside of $\underline{\hat X}$. Let $e$ be such an element. Let $B$ be a basis of $|M|_0$ including $\underline X \setminus x$ but disjoint from $\underline Y \setminus y$. By dualising if necessary, we may suppose without loss of generality that $z \not \in B$. 

If $z$ is not a coloop of $|M|_0 \backslash e$ then by Lemma \ref{lem:extendintospan} there is a circuit $Z$ of $M \backslash e$ such that $\underline{Z^\tl} = \underline {\hat X^\tl} = \underline X$. But by the induction hypothesis applied to $M\backslash e$ its set of circuits satisfies $(C2)$, so by rescaling if necessary we may suppose that $Z^\tl = \hat X^\tl = X$. Since $\underline{Z} \neq \underline{\hat Y}$, we are done as in the above case that $\underline{\hat X} \neq \underline{\hat Y}$.

So we may suppose that $z$ is a coloop of $|M|_0 \backslash e$. Since it is not in $B$, it cannot be a coloop of $|M|_0$. So $z$ and $e$ are in series in $|M|_0$. Thus $t$ and $e$ cannot be in parallel in $|M|_0$. If $t \in B$ then it cannot be a loop in $|M|_0/e$, so we are done by a dual argument to that above. Thus $t \not \in B$ and in particular $t \not \in \underline X$. By an argument like that in the previous paragraph we may suppose that $t$ and $e$ are in series in $|M|_0$.

Since $z$ is not a coloop of $M/t$, by Lemma \ref{lem:extendintospan} there is a circuit $Z$ of $M/t$ such that $\underline{Z^\tl} = \underline {(\hat X\backslash t)^\tl} = \underline X$. By the induction hypothesis applied to $M\backslash e$ its set of circuits satisfies $(C2)$, so by rescaling if necessary we may suppose that $Z^\tl = \hat X^\tl = X$. Let $\hat Z$ be a circuit of $M$ with $\hat Z \setminus t = Z$. Then $\underline{\hat Z^\tl}$ is a vector of $|M|_0$, so it cannot meet the cocircuit $\{t, e\}$ of $|M|_0$ in only the element $t$. Thus $t \not \in \underline{\hat Z^\tl}$ and so $\underline{\hat Z^\tl} = \underline{{\hat X}^\tl} = X$. Since $\underline{\hat Z} \neq \underline{\hat Y}$, we are done as in the above case that $\underline{\hat X} \neq \underline{\hat Y}$.

This completes the proof that $\mathcal{C}_0 \perp_2 \mathcal{D}_0$. Furthermore $\mathcal{C}_0$ satisfies (C1) since $\mathcal{C}$ does. So $\mathcal{C}_0$ is a weak left $R$-signature for $|M|_0$. Similarly $\mathcal{D}_0$ is a weak right $R$-signature for $|M|_0^*$. By Lemma \ref{lem:weaksigs} $\mathcal{C}_0$ is a left $R$-signature for $|M|_0$. 

Next we show that $\mathcal{C}_0 \perp_3 \mathcal{D}_0$. Suppose that $X \in \mathcal{C}_0$ and $Y \in \mathcal{D}_0$ with $|\underline{X} \cap \underline{Y}| \leq 3$. Let $B$ be any basis of $|M|_0$. By Lemma \ref{lem:extendintospan} there is $\hat X \in \mathcal{C}$ with $\underline{\hat X^\tl} = \underline X$ and $\underline{\hat X} \subseteq B \cup \underline X$. Using (C2), by rescaling if necessary we may suppose that ${\hat X}^\tl = X$. By a dual argument there is some $\hat Y \in \mathcal{D}$ with $\hat Y^\tl = Y$ and $\underline Y \subseteq (E \setminus B) \cup \underline{Y}$. Then $|\underline{\hat X} \cap \underline{\hat Y}| = |\underline {X} \cap \underline {Y}| \leq 3$. So $\hat X \perp \hat Y$ and by Lemma  \ref{lem:strat_orth} we have $X \perp Y$.

Thus $\mathcal{C}_0 \perp_3 \mathcal{D}_0$. Applying Lemma \ref{lem:weaksigs} again shows that $M_0 := (E, \mathcal{C}_0)$ is a left $R$-matroid. 
\end{proof}

The matroid $M_0$ of Lemma \ref{lem:H_residue} is the {\em residue matroid} of $M$. Lemma \ref{lem:H_residue}  generalizes Lemma 14 of \cite{Pendavingh2018}. The proof of that lemma, which makes use of quasi-Pl\"ucker coordinates, extends to the general case.

\begin{lemma} \label{lem:strat_space} Let $H=R\rtimes_{U, \psi}\Gamma$, let $M$ be a left $H$-matroid, and let $V,U\in H^E$. If   $V\in \mathcal{V}(M)$ and $V^\tl\in R^E$, then $V^\tl\in \mathcal{V}(M_0)$ and if $U\in \mathcal{U}(M)$ and  $U^\tl\in R^E$, then $U^\tl\in \mathcal{U}(M_0).$\end{lemma}
\proof Suppose $V\in \mathcal{V}(M)$ and $V^\tl \in R^E$. If $V^\tl\not \in \mathcal{V}(M_0)$, then there is a $T\in \mathcal{D}(M_0)$ so that $V^\tl\not\perp T$. Since $\mathcal{D}(M_0)=\minsupp\{ Y^\tl: Y\in\mathcal{D}(M)\}\cap R^E$ by definition of $M_0$, there exists a $Y\in \mathcal{D}(M)$ with $Y^\tl =T$. Then $V^\tl\not\perp Y^\tl$, and it follows that $V\not\perp Y$ by Lemma \ref{lem:strat_space}. This contradicts that $V\in \mathcal{V}(M)$. 

The argument for $U\in \mathcal{U}(M)$ is analogous.
 \endproof

\begin{theorem} \label{thm:strat_perfect} Let $H=R\rtimes_{U, \psi}\Gamma$. If $R$ is perfect, then $H$ is perfect. \end{theorem}
\proof Suppose $R$ is perfect. Let $M$ be a left $H$-matroid, let $V\in \mathcal{V}(M)$ and  $U\in \mathcal{U}(M)$. We need to show that $V\perp U$.  Pick any 
$h\in H^*$ such that  $|V_e|\x |U_e|> |h| \x |U_f|$ for all $e\in \underline{V}\cap\underline{U}$ and $f\in \underline{U}\del\underline{V}$. 
Let $\rho:E\rightarrow H^*$ be determined by $\rho(e)=V_e$ for all $e\in \underline{V}$ and $\rho(e)=h$ otherwise. Since  
$$|\rho U_e|=|V_e\x U_e|=|V_e|\x |U_e|> |h| \x |U_f|=|\rho U_f|$$
for all $e\in \underline{V}\cap\underline{U}$ and $f\in \underline{U}\del\underline{V}$, it follows that $\underline{(\rho U)^\tl}\subseteq \underline{(V \rho^{-1})^\tl}$. Since $U$ is nonzero,  $\underline{(\rho U)^\tl}$ is nonempty, we have $\underline{(V \rho^{-1})^\tl}\cap \underline{(\rho U)^\tl}\neq \emptyset$.
Replacing $M$ with $M^\rho$, we may assume that $\rho\equiv 1$ and hence $\underline{U}\cap\underline{V}\neq \emptyset$. Scaling $V,U$, we may assume that 
$\max_e |V_e|=1$ and $\max_e|U_e|=1$, so that $V^\tl, U^\tl\in R^E$. By Lemma \ref{lem:strat_space}, we have 
$V^\tl\in \mathcal{V}(M_0)$ and $U^\tl\in \mathcal{U}(M_\rho)$. Since $M_0$ is a left $R$-matroid and $R$ is perfect, we have $V^\tl\perp U^\tl$. Since $\underline{U}\cap\underline{V}\neq \emptyset$, it follows that 
 $V\perp U$ by Lemma \ref{lem:strat_orth}, as required.\endproof
Using the classification of stringent skew hyperfields and the fact that the Krasner hyperfield, the sign hyperfield, and skew fields are perfect (Theorem \ref{thm:Kperfect}), we obtain:
\begin{corollary} Let $H$ be a stringent skew hyperfield. Then $H$ is perfect.
\end{corollary}
 
 \subsection{Vector axioms}
Let $H$ be a stringent hyperfield. By the classification of Bowler and Su, we have $H=R\rtimes_{U, \psi}\Gamma$, where $R=\mathbb{K}\text{ or }R=\mathbb{S}$ or $R$ is a skew field. Let $\circ:H\times H\rightarrow H$ be defined as
$$a\circ b =\left\{\begin{array}{ll} 
c&\text{if }a\+ b=\{c\}\\
a&\text{if }a=-b\text{ and }R=\mathbb{K}\text{ or }R=\mathbb{S}\\ 0&\text{if $a=-b$ and  $R$ is a skew field} \end{array}\right.$$
It is then straightforward that $a(b\circ c)= ab \circ ac$ and $(a\circ b)c=ac \circ bc$ for all $a,b,c\in H$, irrespective of $R$. However,
\begin{enumerate}
\item $\circ$ is associative if and only if $R$ is not a skew field;
\item$\circ$ is commutative if and only if $R$ is not the sign hyperfield. 
\end{enumerate}
Thus $H$ is a semi-ring with addition $\circ$ only if $R=\mathbb{K}$.  

In the proof of the main theorem of this section, we will rely on the following property of $\circ$.
\begin{lemma}\label{lem:str_abc} Let $H$ be stringent hyperfield and let  $a,b,c\in H$. If $|a|\geq|b|$, then $c\in a\+(-b)$ implies $a=c\circ b$.\end{lemma}
\proof Suppose $|a|\geq|b|$. If $a=0$, then $b=0$ and hence $c\in a\+(-b)$ implies $c=0$ implies $a=c\circ b$.
If $a\neq 0$, then there are two cases. If $|a|>|b|$, then $c\in a\+(-b)$ implies $c=a$ implies $a=c\circ b$.
If $|a|=|b|$, then we may assume $|a|=|b|=1$. Then $c\in a\+(-b)$ implies $|c|\leq 1$. If $|c|<1$, then  $c\in a\+(-b)$ implies $a=b$ implies $a=c\circ b$. If $|c|=1$, then if 
$R=\mathbb{K}$, then $a=b=c$, so $a=c\circ b$. If $R=\mathbb{S}$, then since $a\neq 0$, we have $c\in a\+(-b)$ implies $a=c=c\circ b$. If $R$ is a skew field, then since $c\neq 0$, we have 
$c\in a\+(-b)$ implies $c=a-b$ implies $a=c+b$ implies $a=c\circ b$.\endproof

\begin{lemma} \label{lem:str_v2} Let $H$ be a stringent hyperfield, and let $M$ be a matroid over $H$. If $V, W\in \mathcal{V}(M)$ and $\underline{V\circ W}=\underline{V}\cup \underline{W}$, then $V\circ W\in\mathcal{V}(M)$.
\end{lemma}
\proof  Let $V, W\in \mathcal{V}(M)$ be such that $\underline{V\circ W}=\underline{V}\cup \underline{W}$. Suppose $V\circ W\not\in \mathcal{V}(M)$. 
Then $(V\circ W)\not\perp Y$ for some $Y\in \mathcal{D}(M)$. 
Rescaling the elements of $M$, we may assume that $V\circ W\in \{0,1\}^E$ and $|Y_e|>|Y_f|$ for all $e\in \underline{V\circ W}$ and $f\not\in \underline{V\circ W}$. 
Scaling $Y$, we may assume that $\max_e |Y_e|=1$ and hence $Y^\tl\in R^E$. 
Since $\underline{V\circ W}=\underline{V}\cup \underline{W}$, it follows that $\max\{|V_e|, |W_e|\}\leq |V_e\circ W_e|=1$ for each $e$. 
If $|V_e|<1$ for all $e$ then $V\circ W=W\in \mathcal{V}(M)$ and we are done. 
So we have $V^\tl\in R^E$ and similarly $W^\tl\in R^E$. 
It follows that $(V\circ W)_e^\tl=V_e^\tl\circ W^\tl_e$ for all $e$. By Lemma \ref{lem:strat_space}, we have $Y^\tl \in \mathcal{U}(M_0)$ and $V^\tl, U^\tl\in \mathcal{V}(M_0)$ . 
Since the statement of the lemma holds true if $H$ is the Krasner hyperfield, the sign hyperfield or a skew field, it follows that $(V\circ W)^\tl=V^\tl\circ W^\tl\in \mathcal{V}(M_0)$. 
Since $R$ is perfect, we have $(V\circ W)^\tl\perp Y^\tl$. By our rescaling, we have $\underline{V\circ W}\cap \underline{Y^\tl}\neq \emptyset$.
Then $(V\circ W)\perp Y$ by Lemma \ref{lem:strat_orth}, a contradiction. 
\endproof

\begin{lemma}\label{lem:sum}Let $H$ be a skew hyperfield, and let $M$ be a left $H$-matroid. If $V^1,\ldots , V^k\in \mathcal{V}(M)$, and $V^1\+\cdots\+V^k=\{V\}$, then $V\in \mathcal{V}(M)$.\end{lemma}
\proof Let $V^1,\ldots , V^k\in \mathcal{V}(M)$, and suppose that $V^1\+\cdots\+V^k=\{V\}$. Consider any $Y\in \mathcal{D}(M)$. By definition of vector, we have $V^i\perp Y$ so that 
$$\bigboxplus_e V_eY_e=\bigboxplus_e\bigboxplus_i V_e^iY_e= \bigboxplus_i\bigboxplus_e V_e^iY_e\supseteq \bigboxplus_i 0\ni 0,$$
and hence $V\perp Y$. Then $V\in \mathcal{V}(M)$.
\endproof

\begin{lemma}\label{lem:decomp} Let $H$ be a stringent hyperfield whose residue hyperfield is the sign hyperfield or a skew field, and let $M$ be a left $H$-matroid. If $V\in \mathcal{V}(M)$, then there are  $X^1,\ldots X^k\in\mathcal{C}(M)$ such that $X^1\+\cdots\+X^k=\{V\}$.\end{lemma}
\proof In the special case that $H$ is itself the sign hyperfield, then the lemma is  equivalently stated as Proposition 3.7.2 of \cite{OM}.   
If $H$ is a skew field, then lemma follows  by induction on $\underline{V}$: take any circuit such that $\underline{X}\subseteq \underline{V}$, and scale $X$ so that $V_e=X_e\neq 0$ for some $e$. Taking $V':=V-X$, we have $\underline{V'}\subseteq\underline{V}\del e$, and hence $V=X^1+\cdots X^k$ by induction. Then $V=X^1+\cdots X^k+X$, as required.

In the general case, let $V\in \mathcal{V}(M)$. Rescaling $M$, we may assume that $V\in\{0,1\}^E$, and deleting any $e\in E\del \underline{V}$ we may assume that $\underline{V}=E$.
Then $V^\tl\in \mathcal{V}(M_0)$ by Lemma \ref{lem:strat_space}. Since the residue $R$ of $H$ is  the sign hyperfield or a skew field and $M_0$ is a left $R$-matroid, there are circuits $T^i\in \mathcal{C}(M_0)$ so that  $T^1\+\cdots\+T^k=\{V^\tl\}$. Let $X^i\in \mathcal{C}(M)$ be such that $(X^i)^\tl=T^i$. For each $e\in E$ we have $\max_i \nu(X_e^i)=1$, so that $\bigboxplus_i X^i_e=\bigboxplus_i T^i_e=\{V_e^\tl\}=\{V_e\}$. Hence $X^1\+\cdots\+X^k=\{V\}$ as required.\endproof

\begin{lemma}\label{lem:multi} Let $H$ be a stringent hyperfield, and let $x_1,\ldots, x_k\in H$. Then $|\bigboxplus_i x_i|=1$ unless $0\in \bigboxplus_i x_i$.\end{lemma}

\begin{theorem}\label{thm:ext} Let $H$ be a stringent hyperfield, let $M$ be a left $H$-matroid on $E$ and let $e\in E$.   Then $\mathcal{V}(M/ e) = \mathcal{V}(M)^e$ and $\mathcal{V}(M\del  e) = \mathcal{V}(M)_e$.
\end{theorem}
\proof If  the residue of $H$ is the Krasner hyperfield, then  the theorem follows Theorem \ref{thm:val_ext}. We may therefore assume that the residue of $H$ is the sign hyperfield or a skew field.

$\mathcal{V}(M/ e) = \mathcal{V}(M)^e$: By Theorem \ref{thm:half_ext} it suffices to show that $\mathcal{V}(M/ e) \subseteq \mathcal{V}(M)^e$. So let $W\in \mathcal{V}(M/ e)$. By Lemma \ref{lem:decomp}, there exist $T^1,\ldots T^k\in\mathcal{C}(M/e)$ so that $T^1\+\cdots\+T^k=\{W\}$. Let $X^i\in \mathcal{C}(M)$ be such that $X^i\del e=T^i$. If $X^1\+\cdots\+X^k=\{V\}$ for some $V\in H^E$, then $V\in \mathcal{V}(M)$ by Lemma \ref{lem:sum}. If not, then we must have $0\in X^1_e\+\cdots\+ X^k_e$ by Lemma \ref{lem:multi}. Consider the vector $V\in H^E$ such that $V\del e=W$ and $V_e=0$. For any $Y\in \mathcal{D}(M)$, we have $X^i\perp Y$, so that $-X^i_eY_e\in \bigboxplus_{f\neq e} X^i_fY_f$. Then 
$$\bigboxplus_f V_fY_f=\bigboxplus_{f\neq e}V_fY_f=\bigboxplus_i \bigboxplus_{f\neq e} X_f^iY_f\supseteq -\bigboxplus_i X^i_eY_e= -(\bigboxplus_i X^i_e)Y_e\ni 0,$$
so that $V\perp Y$. It follows that $V\in \mathcal{V}(M)$.

$\mathcal{V}(M\del  e) = \mathcal{V}(M)_e$: By Theorem \ref{thm:half_ext} it suffices to show that $\mathcal{V}(M\del e) \subseteq \mathcal{V}(M)_e$. So let $W\in \mathcal{V}(M\del e)$. By Lemma \ref{lem:decomp}, there exist $T^1,\ldots T^k\in\mathcal{C}(M\del e)$ so that $T^1\+\cdots\+T^k=\{W\}$. Let $X^i\in \mathcal{C}(M)$ be such that $X^i\del e=T^i$ and $X_e^i=0$. Then $X^1\+\cdots\+X^k=\{V\}$ for some $V\in H^E$ with $V_e=0$, and hence  $V\in \mathcal{V}(M)$ by Lemma \ref{lem:sum}. 
\endproof

\begin{lemma}\label{lem:farkaslike}
Let $H=R\rtimes_{U, \psi}\Gamma$ and let $M$ be a left $H$-matroid on $E$. For any partition of $E$ as $R \dot \cup G \dot \cup B$ there is either
\begin{itemize}
\item A vector $V$ of $M$ such that $|V_e| < 1$ for all $e \in R$, $V_e = 1$ for all $e \in G$ and $V_e = 0$ for all $e \in B$, or
\item A cocircuit $Y$ of $M$ such that $|Y_e| \leq 1$ for all $e \in R \cup G$, $|Y_e| = 1$ for at least one $e \in G$ and $\bigboxplus_{e \in G}Y_e \not \ni 0$.
\end{itemize}
\end{lemma}
\begin{proof}
By induction on $|E|$. If $|R| = 0$ then let $V$ be the indicator function of $G$. The first condition holds if $V$ is a vector of $M$ and the second holds if it is not. 

So say $|R| > 0$ and let $f \in R$. Applying the induction hypothesis to $M\del f$ there is either
\begin{itemize}
\item A vector $V$ of $M \del f$ such that $|V_e| < 1$ for all $e \in R \del \{f\}$, $V_e = 1$ for all $e \in G$ and $V_e = 0$ for all $e \in B$, or
\item A cocircuit $Y$ of $M \del f$ such that $|Y_e| \leq 1$ for all $e \in (R \del\{f\}) \cup G$, $|Y_e| = 1$ for at least one $e \in G$ and $\bigboxplus_{e \in G}Y_e \not \ni 0$.
\end{itemize}
In the first case we are done by Theorem \ref{thm:ext}. In the second there is a cocircuit $Z$ of $M$ with $Z \del f = Y$. We are done if $|Z_f| \leq 1$, hence we may assume $|Z_f| > 1$.

Applying the induction hypothesis to $M/f$ there is either 
\begin{itemize}
\item A vector $V$ of $M/f$ such that $|V_e| < 1$ for all $e \in R \del \{f\}$, $V_e = 1$ for all $e \in G$ and $V_e = 0$ for all $e \in B$, or
\item A cocircuit $Y$ of $M/f$ such that $|Y_e| \leq 1$ for all $e \in (R \del\{f\}) \cup G$, $|Y_e| = 1$ for at least one $e \in G$ and $\bigboxplus_{e \in G}Y_e \not \ni 0$.
\end{itemize}
In the second case there is a cocircuit $Y'$ of $M$ with $Y' \del f = Y$ and $Y'_f = 0$, so we are done. In the first case by Theorem \ref{thm:ext} there is a vector $T$ of $M$ with $T \del f = V$. We are done if $|T_f| < 1$, hence we may assume $|T_f| \geq 1$.

Thus $|T_fZ_f| > 1$ but for any other $e \in E$ we have $|T_eZ_e| \leq 1$ and so $T_e \not \perp Z_e$, which is a contradiction.
\end{proof}

An almost identical proof gives the following.

\begin{lemma}\label{lem:farkaslike2}
Let $H=R\rtimes_{U, \psi}\Gamma$ and let $M$ be a left $H$-matroid on $E$. For any partition of $E$ as $R \dot \cup G \dot \cup B$ there is either
\begin{itemize}
\item A vector $V$ of $M$ such that $|V_e| \leq 1$ for all $e \in R$, $V_e = 1$ for all $e \in G$ and $V_e = 0$ for all $e \in B$, or
\item A cocircuit $Y$ of $M$ such that $|Y_e| < 1$ for all $e \in R$, $|Y_e| \leq 1$ for all $e \in G$, $|Y_e| = 1$ for at least one $e \in G$ and $\bigboxplus_{e \in G}Y_e \not \ni 0$.
\end{itemize}
\end{lemma}

\begin{lemma}\label{lem:str_v3}
Let $H$ be a stringent skew hyperfield. Let $M$ be a left $H$-matroid on $E$ and let $V^1, V^2, \ldots V^k$ be vectors of $M$. Then there is a vector $V \in \bigboxplus_i V_i$. Furthermore for any $e \in E$ with $0 \in \bigboxplus_i V^i_e$ we can find such a $V$ with $V_e = 0$.
\end{lemma}
\begin{proof}
Let $R = \{f \in \bigcup_i \underline{V^i} \colon 0 \in \bigboxplus_iV^i_f\} \del \{e\}$, let $G = \{f \in E \colon 0 \not \in \bigboxplus_iV^i_f\}$ and let $B = E \del (R \cup G)$. By rescaling if necessary, we may suppose that for any $f \in \bigcup_i \underline{V^i}$ we have $\max_i |V^i_f| = 1$ and that for any $f \in G$ the unique element of $\bigboxplus_i V^i_f$ is 1. 

By Theorem \ref{thm:val_axiom}, the residue of $H$ is the Krasner or sign hyperfield or a skew field. Suppose first of all that it is a skew field. Then any $V$ as in the first case of Lemma \ref{lem:farkaslike} would satisfy $V \in \bigboxplus_i V_i$ (and $V_e = 0$), so it suffices to show that there can be no $Y$ as in the second case of Lemma \ref{lem:farkaslike}. So suppose for a contradiction that there is such a $Y$. By rescaling outside $\bigcup_i \underline{V^i}$ if necessary, we can suppose that for $f \not \in \bigcup_i \underline{V^i}$ we have $|Y_f| \leq 1$. Thus if there is any $f$ with $|Y_f| >1$ then the only such $f$ is $e$ and we have $e \in \bigcup_i \underline{V^i}$. But then picking some $i$ with $|V^i_e| = 1$ we would have $V^i \not \perp Y$, which is impossible. So $|Y_f| \leq 1$ for all $f \in E$. For each $i$ we have $(V^i)^\tl \perp Y^\tl$ and so $\sum_i(V^i)^\tl \perp Y^\tl$, so that $\sum_{f \in G}Y^\tl_f = 0$, contradicting $\bigboxplus_{f \in G}Y^f \not \ni 0$. 

Now consider the case that the residue of $H$ is the Krasner or sign hyperfield. Then any $V$ as in the first case of Lemma \ref{lem:farkaslike2} would satisfy $V \in \bigboxplus_i V_i$ (and $V_e = 0$), so it suffices to show that there can be no $Y$ as in the second case of Lemma \ref{lem:farkaslike2}. But if there were such a $Y$ then we would obtain
$$\bigboxplus_{f \in G}Y_f = \bigboxplus_{f \in G} \bigboxplus_i V^i_fY_f = \bigboxplus_{f \in E \del e} \bigboxplus_i V^i_fY_f = \bigboxplus_i \bigboxplus_{f \in E \del e} V^i_fY_f \supseteq \bigboxplus_i (-V^i_eY_e) = -Y_e\bigboxplus_iV^i_e\ni 0,$$
a contradiction.
\end{proof}

 \ignore{
 \begin{lemma}\label{lem:span} Let $H=R\rtimes_{U, \psi}\Gamma$, where $R=\mathbb{S}$  or $R$ is a skew field, let $M$ be a left $H$-matroid on $E$, let  $X^1,\ldots , X^k\in \mathcal{V}(M)$ and $e\in E$.
If $|\bigboxplus_i X^i_e|\neq 1$, then there exist $Y^i\in \mathcal{V}(M)$ such that $\max_i |Y^i_e|<\max_i |X^i_e|$ and $\bigboxplus_i Y^i \subseteq \bigboxplus_i  X^i.$\end{lemma}
\proof  If $H=R$, then $\Gamma$ is a trivial group, and then the condition that $\max_i |Y^i_e|<\max_i |X^i_e|$ amounts to $Y_e^i=0$ for all $i$. If $H=R$ is the sign hyperfield, we may assume $k=2$ by omitting all $X^i$ except one with $X^i_e= 1$ and one with $X^j_e=-1$.  Then the lemma  follows by applying  (V3) for oriented matroids to $V=X^1, W=X^2, e$. If $H=R$ is a skew field, then $Y=\sum_i X^i$ satisfies the condition of the Lemma.

We use induction on $|E|$. If there is an $f\in E$ so that $X^i_f=0$ for all $i$, then by Theorem \ref{thm:half_ext} we have $X^i\del f\in \mathcal{V}(M\del f)$. Then  $T^i:=(X^i\del f) \in  \mathcal{V}(M\del f)$, and by induction there are $Z^i\in \mathcal{V}(M\del f)$ so that $\bigboxplus_i Z^i \subseteq \bigboxplus_i  T^i$ and $\max_i |Z^i_e|<\max_i |T^i_e|$. By Theorem \ref{thm:ext}, there are vectors $Y^i\in \mathcal{V}(M)$ with $Y^i_f=0$ and $Y^i\del f=Z^i$. Then the $Y^i$ satisfy the conditions of the Lemma.

So for each $f\in E$ there is some $i$ so that $X^i_f\neq 0$.  Rescale $M$ so that $\max_f |X_f^i|=1$ for all $f\in E$. 
We have $T^i=(X^i)^\tl\in\mathcal{V}(M_0)$, and since the lemma holds true for matroids over the residue hyperfield $R$, there exist $Z^i\in \mathcal{V}(M_0)$ so that $\bigboxplus_i^R Z^i\subseteq \bigboxplus_i^R T^i $ and $Z_e^i=0$ for all $i$. We 
we may assume without loss of generality that $Z^i\in \mathcal{C}(M_0)$ by using Lemma \ref{lem:decomp}. Let $Y^i\in \mathcal C(M)$ be such that $(Y^i)^\tl=Z^i$. Then $\bigboxplus_i Y^i \subseteq \bigboxplus_i  X^i$ and $\max_i |Y^i_e|<\max_i |X^i_e|$, as required.
\endproof

\begin{lemma} \label{lem:str_v3}Let $H$ be a stringent hyperfield, and let $M$ be a left $H$- matroid on $E$. If $V, W\in \mathcal{V}(M), e\in E$ such that $V_e=-W_e\neq 0$, then there is a $Z\in \mathcal{V}(M)$ such that $Z\in V\+ W$ and $Z_e=0$.\end{lemma}
\proof By Theorem \ref{thm:val_axiom}, it suffices to prove the lemma for the case that the residue $R$ of $H$ is the sign hyperfield or a skew field. Let  $V, W\in \mathcal{V}(M)$ be such that $V_e=-W_e\neq 0$. We show that there is a $Z\in \mathcal{V}(M)$ such that $Z\in V\+ W$ and $Z_e=0$, by induction on $|E|$.

A a {\em good collection} is a finite sequence $(X^i)_i$ such that
$$X^i\in \mathcal{C}(M)\text{ for each }i,~ \bigboxplus_i X^i\subseteq V\+ W \text{ and }0\in \bigboxplus_i X^i_e.$$
Good collections exist. By Lemma \ref{lem:decomp}, there exist $X^i\in \mathcal{C}(M)$ so that $\{V\}=X^1\+\cdots \+ X^{k'}$ and $\{W\}=X^{k'+1}\+\cdots \+ X^k$, and then $(X^i)_{i=1}^k$ is a good collection. For later use, we fix $Y:=(X^j_e)^{-1}X^j$ for any $j$ so that $|X^j_e|=|V_e|$. Then $Y_e=1$ and  $|V_e||Y_f|\leq \max\{|V_f|, |W_f|\}$ for all $f\in E$. 
 
We first show that there exists a good collection $(X^i)$ so that $X^i_e=0$ for all $e$. Consider the set of values 
$$S:=\{|X_e|: X\in \mathcal{C}(M), ~|X_f|=\max\{|V_f|, |W_f|\}\text{ for some }f\in \underline{V}\cup\underline{W}\}.$$
Then $S$ is finite, as  by (C2) each circuit of the underlying matroid $\underline{M}$ contributes at most one value to $S$. We claim that for each good collection  $(X^i)_i$ there is a good collection $(Y^i)_i$ so that $\max_i |X^i_e|\geq \max_i |Y^i_e|\in S$. Suppose $(X^i)_{i=1}^k$ is a shortest sequence for which this fails. Then $\max_i |X^i_e|\not \in S$. Rearranging the $X^i$, we may assume that  $|X_e^j|=\max_i |X^i_e|$ if and only if $j>t$. Then for each $j>t$, we have  $|X^j_f|<\max\{|V_f|, |W_f|\}$ for all $ f\in \underline{V}\cup\underline{W}$, and hence we have  
$\bigboxplus_{i\leq t} X^i\subseteq V\+ W$. Since $0\in \bigboxplus_i X^i=\bigboxplus_{i>t}X^i$, we have $k-t\geq 2$. Pick any $y\in \bigboxplus_{i\leq t} X^i_e$.  Then $|y|<\max_i |X^i_e|\leq |V_e|$ and hence $|yY_f|<\max\{|V_f|, |W_f|\}$ for all $f\in E$, so that the sequence $(Z^i)_i:=(X^1, \ldots , X^t, -yY)$ is a good collection of length $t+1<k$.
By our choice of $(X^i)$, there is a good collection $(Y^i)$ so that $\max_i |X^i_e|>|y|=\max_i |Z^i_e| \geq \max_i |Y^i_e|\in S$, a contradiction.

Hence, the minimum of $\max_i |X^i_e|$ over all good collections $(X^i)$ takes value in the finite set $S$. Let $(X^i)$ attain the minimum.
If $|\bigboxplus_i X^i_e|\neq 1$, then by Lemma \ref{lem:span} there is a collection $(Y^i)$ with $Y^i\in \mathcal{V}(M)$ such that $\max_i |Y^i_e|<\max_i |X^i_e|$ and $\bigboxplus_i Y^i \subseteq \bigboxplus_i  X^i.$ Using Lemma \ref{lem:decomp}, we may assume that each $Y^i\in \mathcal C(M)$. By our choice of $(X^i)$, we cannot have $0\in \bigboxplus_i Y^i_e$, and so $\bigboxplus_{i\leq k} Y^i_e=\{y\}$. Then extending $(Y^i)$ with $-yY$ yields a good collection as before, which violates the choice of $(X^i)$. So $X^i_e=0$ for all $i$, as required.
 
Let $(X^i)_{i=1}^k$ be a shortest good collection with $X^i_e=0$ for all $i$. We claim that $k=1$. If not, consider $X^{k-1}$ and $X^k$. If $X^{k-1}\+X^k=\{Z\}$, then $Z\in \mathcal{V}(M)$ by Lemma \ref{lem:sum}, and otherwise there is an $f\neq e$ so that $X^{k-1}_f=-X^k_f$. Then by our induction hypothesis, there exists a $T\in \mathcal{V}(M\del e)$ so that  $T_f=0$ and $T\in (X^{k-1}\del e) \+ (X^k\del e)$. By Theorem \ref{thm:ext}, there is a $Z\in \mathcal{V}(M)$ so that $Z_e=0$ and $Z\del e=T$.
In either case, we have $X^1\+\cdots\+X^{k-2}\+Z\subseteq \bigboxplus_i X^i\subseteq V\+ W$, so that $(X^1, \ldots,X^{k-2},Z)$ is a shorter good collection, a contradiction. 
Hence $k=1$, and taking $Z=X^1$ we have $Z\in \mathcal{V}(M)$, $Z_e=0$ and $Z\in V\+ W$, as required.
\endproof

}

\begin{theorem} Let $H$ be a stringent skew hyperfield. Let $E$ be a finite set, and let $\mathcal{V}\subseteq H^E$. There is a left $H$-matroid $M$ such that $\mathcal{V}=\mathcal{V}(M)$ if and only if
\begin{enumerate}
\item[(V0) ] $0\in \mathcal{V}$.
\item[(V1) ] if $a\in H$ and $V\in \mathcal{V}$, then  $aV\in \mathcal{V}$.
\item[(V2)$'$] if $V, W\in \mathcal{V}(M)$ and $\underline{V\circ W}=\underline{V}\cup \underline{W}$, then $V\circ W\in\mathcal{V}(M)$. 
\item[(V3) ] if $V, W\in \mathcal{V}, e\in E$ such that $V_e=-W_e\neq 0$, then there is a $Z\in \mathcal{V}$ such that $Z\in V\+ W$ and $Z_e=0$.
\end{enumerate}
Then $\mathcal{C}(M)=\minsupp(\mathcal{V}\del\{0\})$.
\end{theorem}
\proof Sufficiency: Suppose $\mathcal{V}$ satisfies (V0),(V1),(V2)$'$, and (V3). Let $\mathcal{C}:=\minsupp (\mathcal{V}\del\{0\})$.
Then $\mathcal{C}$ satisfies (C1) by (V1). To see (C2), let $X,Y\in\mathcal{C}$ be such that $\underline{X}\subseteq \underline{Y}$. If $Y\neq a X$ for all $a\in H^\star$, then scaling $X$ so that $Y_e=-X_e$ for some $e\in\underline{X}$, we have $X\neq -Y$. By (V3), there is a $Z\in \mathcal{V}$ so that $Z_e=0$ and $Z\in X\+ Y$, and since $X\neq -Y$ we have $Z\neq 0$. Then $\emptyset\neq \underline{Z}\subseteq \underline{Y}\del e$, contradicting that $Y\in \mathcal{C}$. We show that $\mathcal{C}$ satisfies the modular circuit elimination axiom (C3).
If $X, Y\in \mathcal{C}$ are a modular pair, $X_e=-Y_e$, then by (V3) there exists a $Z\in \mathcal{V}$ such that $Z\in X\+ Y$ and $Z_e=0$.
 If $Z\not\in \mathcal{C}$, then there exists a $Z'\in\mathcal{C}$ so that $\underline{Z'}$ is a proper subset of $\underline{Z}$. 
 Applying (V3) to $Z, Z'$, $f\in\underline{Z'}\subseteq \underline{Z}$ then  implies the existence of a $Z''\in\mathcal{C}$ such that $\underline{Z''}$ is  contained in $\underline{Z}\del f$. 
 Then the existence of $Z', Z''\in\mathcal{C}$ would contradict the modularity of the pair $X,Y$ in $\mathcal{C}$, since $\underline{Z'}\cup\underline{Z''}\subseteq \underline{Z}\subseteq \underline{X}\cup\underline{Y}\del\{e\}$. Hence, we have $Z\in \mathcal{C}$. 
 This proves that $\mathcal{C}$ also satisfies modular circuit elimination, so that $\mathcal{C}=\mathcal{C}(M)$ for some left $H$-matroid $M$. 
 We show that $\mathcal{V}=\mathcal{V}(M)$. 
To see $\mathcal{V}\subseteq \mathcal{V}(M)$, suppose $V\in \mathcal{V}\del\mathcal{V}(M)$ and $V$ has minimal support among all such vectors. Let $X\in \mathcal{C}$ be any vector with $\underline{X}\subseteq\underline{V}$. 
 Scale $X$ so that $|X_f|\leq |V_f|$ for all $f\in E$, with $X_e=V_e$ for some $e$. 
 Then applying (V3) to $V,-X,e$ yields a vector $Z$ such that  $Z\in V\+ (-X)$. Then $V=Z\circ X$ by Lemma \ref{lem:str_abc}. We have $X\in \mathcal{V}(M)$ as $X\in \mathcal{C}(M)$ and $Z\in \mathcal{V}(M)$ by minimality of $V$. Then $V=Z\circ X \in \mathcal{V}(M)$ as (V2)$'$ holds for $\mathcal{V}(M)$ by Lemma \ref{lem:str_v2}.
 That $\mathcal{V}(M)\subseteq \mathcal{V}$ follows in the same way, since (V2)$'$ holds for $\mathcal{V}$ by assumption and  (V3) holds for $\mathcal{V}(M)$ by Lemma \ref{lem:str_v3}.

Necessity: If $\mathcal{V}=\mathcal{V}(M)$, then (V0),(V1) are clear, (V2)$'$ is  Lemma \ref{lem:str_v2} and (V3) follows from  Lemma \ref{lem:str_v3}.
\endproof
 We note that if the residue $R$ of $H$ is the Krasner or sign hyperfield, then $\underline{V\circ W}=\underline{V}\cup \underline{W}$ for all $V, W\in H^E$, so that then condition (V2)$'$ may be simplified to (V2) as in Theorem \ref{thm:val_axiom}. If $R$ is a skew field, then (V2)$'$ is equivalent to 
\begin{enumerate}
\item[(V2)$''$] if $V, W\in \mathcal{V}(M)$ and $V\+W=\{U\}$, then $U\in\mathcal{V}(M)$. 
\end{enumerate}

A minor adaptation of the proof of Theorem \ref{thm:val_circ} yields the following characterization.
\begin{theorem}Let $H$ be a stringent skew hyperfield with residue $\mathbb{S}$. Let $E$ be finite set and let $\mathcal{C}\subseteq  H^E$. Then $M=(E,\mathcal{C})$ is a  left $H$-matroid  if and only if (C0), (C1), (C2) and 
\begin{itemize}
\item [(C3)$'$] for any $X, Y\in \mathcal{C}, e,f\in E$ such that $X_e=-Y_e\neq 0$ and $|X_f|>|Y_f|$, there is a $Z\in \mathcal{C}$ such that $Z_e=0$, $Z_f=X_f$, and $|Z_g|< |X_g \circ Y_g|$ or $Z_g\in X_g\+ Y_g$ for all $g\in E$.
\end{itemize}
\end{theorem}
\ignore{
\proof Necessity: Suppose that $M=(E,\mathcal{C})$ is a  left $H$-matroid. Then (C0), (C1), (C2) hold by definition, and we show (C3)$'$. So assume that $X, Y\in \mathcal{C}, e,f\in E$ are such that $X_e=-Y_e\neq 0$, and $X_f>Y_f$. By the vector axiom (V3), there exists a $V\in \mathcal{V}(M)$ such that $V\in X\+ Y$ and $V_e=0$. As $|X_f|>|Y_f|$, we have $V_f=X_f$. By Lemma \ref{lem:decomp}, there exist circuits $Z^1,\ldots, Z^k\in \mathcal{C}$ so that $\{V\}=Z^1\+\cdots\+ Z^k$. Pick $i$ so that $V_f=Z^i_f$ and define $Z:=Z^i$.  Then $Z\in \mathcal{C}$, $Z_e=0$ since $V_e=0$,  $Z_f=V_f=X_f$, and $|Z|\leq |V|\leq |X\circ Y|$. Moreover, if $|Z_g|=|X_g\circ Y_g|$, then  $|Z_g|=|V_g|$, and then $Z_g=V_g\in X_g\+ Y_g$, as required.

Sufficiency: Suppose (C0), (C1), (C2), (C3)$'$ hold for $M=(E, \mathcal{C})$. To show that $M$ is a  left $H$-matroid it suffices to show (C3). So let $X,Y\in \mathcal{C}$ be modular circuits so that $X_e=-Y_e$. Pick any $f\in \underline{X}\del\underline{Y}$. By (C3)$'$, there exists a $Z\in \mathcal{C}$ such that $Z_e=0$, $Z_f=X_f$, and $|Z_g|< |X_g \circ Y_g|$ or $Z_g\in X_g\+ Y_g$ for all $g\in E$. If $Z\in X\+ Y$ then we are done, so let $h\in E$ be such that $Z_h\not\in X_h\+Y_h$. Then $|Z_h<|X_h\circ Y_h|$ and $X_h\neq -Y_h$. If $|X_h|>|Y_h|$, then apply (C3)$'$ to $(X,Y, e, h)$ to find a $Z'\in \mathcal{C}$ such that $Z'_e=0$,  $Z'_h=X_h$, and $Z'\leq |X\circ Y|$. Since $X,Y$ are modular and $\underline{Z}\cup\underline{Z'}\subseteq \underline{X}\cup\underline{Y}\del e$, we have  $\underline{Z}=\underline{Z'}$, and hence $Z'=\alpha Z$ for some $\alpha\in H^\star$ by (C2). Then $Z_h< X_h=Z'_h=\alpha Z_h$, and $Z_f=X_f\geq Z'_f=\alpha Z_f$, a contradiction. If $|X_g|<|Y_g|$, we apply (C3)$'$ to $(Y, X, e, h)$ to obtain a $Z'$ with 
$Z_h< Y_h=Z'_h=\alpha Z_h$, and $Z_f=X_f\geq Z'_f=\alpha Z_f$, which again yields a contradiction. Finally, if $|X_h|=|Y_h|$ then $X_h=-Y_h$, and then $Z_h\in X_h\+Y_h$.\endproof
}

For the circuits $\mathcal{C}$ of a matroid over a stringent skew hyperfield $H$ whose residue is a skew field, combining (V3) and Lemma \ref{lem:decomp} evidently yields:
\begin{itemize}
\item [(C3)$''$] for any $X, Y\in \mathcal{C}, e\in E$ such that $X_e=-Y_e\neq 0$, there is a $V\in X\+ Y$ and $Z^i\in \mathcal{C}$ such that $V_e=0$ and $Z^1\+\cdots \+Z^k=\{V\}$. 
\end{itemize}
For such hyperfields $H$, we could not imagine an axiom which is sufficiently strong to characterize matroids over $H$, but also more like (C3)$'$ in that it claims the existence of just a single circuit $Z$.

\section{Related work}
In this section we explain in more detail the connections between our work and the examples mentioned at the start of the introduction. Fundamental to these connections is the notion of {\em push-forwards} of matroids over skew hyperfields along skew hyperfield homomorphisms. If $f \colon H \to K$ is a skew hyperfield homomorphism and $X \in H^E$ then we define $f_*X \in K^E$ by $e \mapsto f(X(e))$. If $\mathcal C$ is a left $H$-matroid then the {\em push-forward} $f_*(\mathcal{C}) := \{\alpha \cdot (f_*X) \colon \alpha \in K^{*}, X \in \mathcal{C}\}$ is again a left $K$-matroid.

\subsection{Linear spaces over valued fields} Let $K$ be a skew field with a valuation  $|.|:K\rightarrow \Gamma_{\max}$.    Then by Krasner's Theorem (cited as Theorem \ref{thm:krasner} in this paper) and in particular Lemma \ref{lem:valfield}, there is a skew hyperfield $K/G$ and commuting hyperfield homomorphisms
\begin{center}
\begin{tikzcd}
  K \arrow[r, "\nu"] \arrow[dr, "|.|"'] & K/G\arrow[d,"|.|"]\\
  & \Gamma_{\max}
\end{tikzcd}
\end{center}
where $G:=\{1+k: |k|<1, k\in K\}$. By Lemma \ref{lem:valfield}, the skew hyperfield $K/G$ is a stringent hyperfield, and its residue $R$ is mapped to the residue $\mathbb{K}$ of $\Gamma_{\max}$ by the homomorphism $|.|:K/G\rightarrow \Gamma_{\max}$ of the diagram. Taking the push-forward of $M$ along the homomorphism $|.|:K\rightarrow \Gamma_{\max}$ produces a valuated matroid $|M|$, and the homomophism $\nu:K\rightarrow K/G$ likewise gives a push-forward $\nu_* M$. We then have $|M|=|\nu_* M|$. 

Given a valued field $K$ and valuated matroid $N$, it is generally difficult to decide if $N=|M|$ for some $K$-matroid $M$. Brandt argues in  Proposition 2.8 of \cite{Brandt2019} that if a valuated matroid $N$ arises as the push-forward $N=|M|$ of a $K$-matroid, then the residue matroid $N_0=|M|_0$ must be linear over the residue field $R$ of $K$, and that indeed the residue of each matroid that arises by rescaling $N=|M|$ must be linear over $R$. 
The combined necessary condition does not take into account that the several linear representations over $R$ of these residue matroids must agree if they derive from a common source $M$. 
A stronger necessary condition arises by noting that if $N=|M|$, then there exists a $K/G$-matroid $M'$ so that $|M'|=N$, namely $M'=\nu_* M$. Such a matroid $M'$ is uniquely determined by linear representations of its residue matroids over $R$, but not each collection of linear representations gives a $K/G$-matroid $M'$.

\subsection{Singularities}

In \cite{Juergens2018}, J\"urgens classifies the singularities of real plane tropical curves. Much of the basic framework he works with can be profitably understood in terms of push-forwards of matroids along maps between certain stringent skew hyperfields. J\"urgens works with the field $\mathbb{R}\{\{t\}\}$ of real Puiseux series, which are power series with real coefficients in the indeterminate $t$ with rational exponents that have a common denominator and are bounded below. For any skew hyperfield $H$ and ordered group $\Gamma$ there is a canonical exact sequence of multiplicative groups $1 \to H^* \to H^* \times \Gamma \xrightarrow{\pi_2} \Gamma \to 1$, and we denote the corresponding skew hyperfield $H \rtimes_{H^* \times \Gamma, \pi_2} \Gamma$ by $\semi{H}{\Gamma}$. Consider the following commutative diagram of hyperfields:

\begin{center}
\begin{tikzcd}
  & \mathbb{R} \ar[dl, "\phi"'] \ar[d] \ar[r, "\sigma"] & \mathbb{S} \ar[d] \ar[dd, "\tau", bend left]\\
 \mathbb{R}\{\{t\}\} \ar[r, "\nu"'] \ar[drr, "\rho"', bend right = 50] & \semi{\mathbb{R}}{\mathbb{Q}} \ar[d] \ar[r] & \semi{\mathbb{S}}{\mathbb{Q}} \ar[d] \\
  & \semi{\mathbb{R}}{\mathbb{R}} \ar[r] & \semi{\mathbb{S}}{\mathbb{R}} \\
  \end{tikzcd}
\end{center}

where the map $\nu$ is given as in Lemma \ref{lem:valfield} for the standard valuation on $\mathbb{R}\{\{t\}\}$, the other horizontal maps are induced from the map sending a real number to its sign and the vertical maps are induced from the embeddings $1 \to \mathbb{Q} \to \mathbb{R}$ of ordered groups. 

A linear variety $\mathcal{V}$ over $\mathbb{R}\{\{t\}\}$ is simply a linear subspace of $\mathbb{R}\{\{t\}\}^E$ for some $E$, and so can be identified with a left $\mathbb{R}\{\{t\}\}$-matroid. The real tropicalisation of such a variety as defined in \cite[Definition 1.2.6]{Juergens2018} is then given by (the set of vectors of) its push-forward $\rho_*(\mathcal{V})$. The push-forward along $\tau$ also plays an important role: for an oriented matroid $M$, the oriented initial matroids discussed in \cite[Remark 2.2.3]{Juergens2018} are simply the residue matroids of $\tau_*(M)$. Similarly the signed Bergman fan $\mathcal{B}^s$ of an oriented matroid $M$ on $E$ with respect to $s \in \{\pm\}^E$ as defined in \cite[Definition 2.2.11]{Juergens2018} is simply $\{(\pi_2)_* X \colon X \in \tau_*(M), (\pi_1)_*X \circ s = s\}$. Given these identifications, \cite[Theorem 2.2.16]{Juergens2018}, which says how for a linear variety $\mathcal{V}$ in the image of $\phi_*$ the signed Bergman fans of the associated oriented matroid can be defined from the tropicalisation of $\mathcal{V}$, reduces to the statement $\rho_* \cdot \phi_* = \tau_* \cdot \sigma_*$, which is immediate from the commutativity of the above diagram. 

Our account gives a more generally applicable point of view on many of the results in the paper. For example, \cite[Lemma 3.3.4]{Juergens2018} and \cite[Lemma 3.3.5]{Juergens2018}, whose proofs together constitute a fifth of the paper, follow straightforwardly from Lemma \ref{lem:str_v2} applied to $\semi{\mathbb{S}}{\mathbb{R}}$-matroids in the image of $\tau_*$.

\subsection{Field extensions} If $K, L$ are fields so that $K\subseteq L$, then a list of elements $x_1,\ldots, x_k\in L$ is {\em algebraically dependent over $K$} if there exists a nonzero polynomial $p\in K[X_1,\ldots, X_k]$ so that $p(x_1,\ldots,x_k)=0$.
Any collection $x\in L^E$ gives rise to a matroid on ground set $E$ in which a subset $F\subseteq E$ is dependent if and only if $(x_f: f\in F)$ is algebraically dependent over $K$. Thus the pair $K,x$ determines an {\em algebraic matroid} $M(K,x)$.

A $K$-{\em derivation} of $L$ is a function $D: L\rightarrow L$ so that 
\begin{itemize}
\item[(D0)] $D(x)=0$ for all $x\in K$
\item[(D1)] $D(x+y)=D(x)+D(y)$ for all $x,y\in L$
\item[(D2)] $D(xy)=D(x)y+xD(y)$ for all $x,y\in L$.
\end{itemize}
A collection $x\in L^E$ determines a linear space
$$Der(K,x):=\left\{(D(x_e): e\in E): D\text{ is a $K$-derivation of $L$}\right\}\subseteq L^E.$$
Being a linear space over $L$, $Der(K,x)$ is the set of covectors of an $L$-matroid, which we will denote as $D(K, x)$. 

The matroids $M(K,x)$ and $D(K,x)$ have ground set $E$, and both matroids have the same rank. Ingleton showed that if the characteristic of $K$ is 0, then $M(K,x)$ is the underlying matroid of $D(K,x)$. For fields $K$ of positive characteristic $p$ this need not be the case, but it is known that the matroid underlying $D(K,x)$ is a weak image of $M(K,x)$. That is, if $F\subseteq E$ is independent in $D(K,x)$ then $F$ is independent in $M(K,x)$, but the converse need not be true. Moreover, if we denote the Frobenius map as $\sigma: x\mapsto x^p$ and put
$$\sigma^\rho(x):=\left(\sigma^{\rho(e)}(x_e): e\in E\right)$$
for $\rho: E\rightarrow\Z$ (we assume here that $L$ is perfect), then $M(K,\sigma^\rho(x))=M(K,x)$ for all $\rho$, but $D(K,\sigma^\rho(x))$ in general does not equal $D(K,x)$.

By results of \cite{Pendavingh2018}, the exact relation between $M(K,x)$ and $D(K,x)$ can be clarified using a matroid over a stringent hyperfield, as follows. There exists a stringent skew hyperfield $L^\sigma$ with residue $L$, and a left $L^\sigma$-matroid $M^\sigma(K, x)$, so that
\begin{itemize}
\item $M(K,x)$ is the matroid underlying $M^\sigma(K, x)$;
\item $D(K,x) = M^\sigma(K, x)_0$, the residue matroid; and
\item $M^\sigma(K,\sigma^\rho(x))$ arises from $M^\sigma(K, x)$ by a rescaling according to $\rho$.
\end{itemize}
The hyperfield $L^\sigma$ arises by Krasner's construction from the quotient field $L(T, \sigma)$ of the Ore ring $L[T, \sigma]$, as in Lemma \ref{lem:valfield}. The Ore ring $L[T, \sigma]$ has the same elements as the usual polynomial ring $L[T]$, but the non-commutative multiplication with the variable $T$ is instead determined by $xT=T\sigma(x)$. 

We refer to \cite{Pendavingh2018} for a more complete account.

\ignore{
The key step in establishing this fact is elimination of a variable between two polynomials, as follows.

For a polynomial $p\in K[X_e: e\in E]$, let $\underline{p}$ denote the smallest $F\subseteq E$ so that $p\in K[X_e: e\in F]$,  the {\em variable support} of $p$. 
\begin{lemma} Let $p, q\in K[X_e: e\in E]$, and let $e\in \underline{p}\cap\underline{q}$. Then there exist polynomials $s,t\in K[X_e: e\in E]$ with $\deg_e(s)<\deg_e(q)$ and $\deg$
\end{lemma}

 For a polynomial $p\in K[X_e: e\in E]$, the {\em support of $p$} is a vector $\underline{p}\in \mathbb{K}^E$ such that $\underline{p}_e\neq 0$ exactly if the variable $X_e$ appears in a monomial with nonzero coefficient in $p$.
\begin{lemma} Suppose $K, L$ are fields so that $K\subseteq L$, let $E$ be a finite set and let $x\in L^E$. Let 
$f:K[X_e: e\in E]\rightarrow L$ be the ring homomorphism so that $f: X_e\mapsto x_e$. Then 
$$\mathcal{V}:=\{ \underline{p} : p\in \ker(f)\}$$
is the set of vectors of a matroid over $\mathbb{K}$.
\end{lemma}
\proof We verify the vector axioms. It is evident that (V0) $0=\underline{0}\in \mathcal{V}$, and axiom (V1) is trivial over $\mathbb{K}$. To see (V2), note that if $V, W\\mathcal{V}$, then $V=\underline{p}, W=\underline{q}$ for some $p, q\in\ker(f)$. Then for each $r$ we have $p+rq\in\ker(f)$, and for some $r$ we have $V\circ W=\underline{p+rq}\in \mathcal{V}$. Finally, if $V, W\mathcal{V}$ and $V_e=W_e\neq 0$, again consider $p, q\in\ker(f)$ so that $V=\underline{p}, W=\underline{q}$. 
Then variable $X_e$ appears in both $p$ and $q$, and there exist nonzero polynomials $s,t$ so that $sp+tq$ eliminates $X_e$. 
}


\bibliographystyle{alpha}
\bibliography{math}

\end{document}